\documentclass[accepted,onefignum,onetabnum]{siamonline190516}

\usepackage{lipsum}
\usepackage{amsfonts}
\usepackage{graphicx}
\usepackage{epstopdf}
\ifpdf
  \DeclareGraphicsExtensions{.eps,.pdf,.png,.jpg}
\else
  \DeclareGraphicsExtensions{.eps}
\fi

\usepackage{algorithm,algcompatible,algpseudocode}
\usepackage{amsmath}

\usepackage{amssymb}
\usepackage{dirtytalk}
\usepackage{enumitem}
\setlist[enumerate]{leftmargin=.5in}
\setlist[itemize]{leftmargin=.5in}

\newsiamremark{remark}{Remark}
\newsiamremark{hypothesis}{Hypothesis}
\crefname{hypothesis}{Hypothesis}{Hypotheses}
\newsiamthm{claim}{Claim}

\headers{The Practicality of Stochastic Optimization in Imaging Inverse Problems}{J.Tang, K. Egiazarian, M. Golbabaee, M.Davies}

\title{The Practicality of Stochastic Optimization in Imaging Inverse Problems\thanks{October 2019. Some preliminary results of this work was published in \cite{tang2019limitation}.}}

\author{Junqi Tang\thanks{School of Engineering, University of Edinburgh, UK
  (\email{J.Tang@ed.ac.uk}).}
  \and Karen Egiazarian\thanks{Noiseless Imaging Ltd, Finland
  (\email{karen.egiazarian@noiselessimaging.com}).}
    \and Mohammad Golbabaee\thanks{Department of Computer Science, University of Bath
  (\email{mg2105@bath.ac.uk}).}
\and Mike Davies\thanks{School of Engineering, University of Edinburgh, UK
  (\email{Mike.Davies@ed.ac.uk}).}  
  }

\usepackage{amsopn}

\newcommand{\X}{\mathcal{X}}
\newcommand{\A}{\mathcal{A}}
\newcommand{\E}{\mathbb{E}}

\let \oldsection \section
\renewcommand{\section}{\vspace{3ex plus 1ex}\oldsection}

\newcommand{\BEAS}{\begin{eqnarray*}}
	\newcommand{\EEAS}{\end{eqnarray*}}
\newcommand{\BEA}{\begin{eqnarray}}
\newcommand{\EEA}{\end{eqnarray}}

\newcommand{\BEQ}{\begin{equation}}
\newcommand{\EEQ}{\end{equation}}
\newcommand{\BIT}{\begin{itemize}}
	\newcommand{\EIT}{\end{itemize}}
\newcommand{\BNUM}{\begin{enumerate}}
	\newcommand{\ENUM}{\end{enumerate}}

\newcommand{\BA}{\begin{array}}
	\newcommand{\EA}{\end{array}}

 \numberwithin{dummy}{section}

\newcommand{\mr}{\mathrm}
\newcommand{\mb}{\mathbb}
\newcommand{\mc}{\mathcal}

\ifpdf
\hypersetup{
  pdftitle={The Practicality of Stochastic Optimization in Imaging Inverse Problems},
  pdfauthor={Junqi Tang}
}
\fi

\begin{document}

\maketitle

 \begin{abstract}
 In this work we investigate the practicality of stochastic gradient descent and recently introduced variants with variance-reduction techniques in imaging inverse problems. Such algorithms have been shown in the machine learning literature to have optimal complexities in theory, and provide great improvement empirically over the deterministic gradient methods. Surprisingly, in some tasks such as image deblurring, many of such methods fail to converge faster than the accelerated deterministic gradient methods, even in terms of epoch counts. We investigate this phenomenon and propose a theory-inspired mechanism for the practitioners to efficiently characterize whether it is beneficial for an inverse problem to be solved by stochastic optimization techniques or not.  Using standard tools in numerical linear algebra, we derive conditions on the spectral structure of the inverse problem for being a suitable application of stochastic gradient methods. Particularly, we show that, for an imaging inverse problem, if and only if its Hessain matrix has a fast-decaying eigenspectrum, then the stochastic gradient methods can be more advantageous than deterministic methods for solving such a problem. Our results also provide guidance on choosing appropriately the partition minibatch schemes, showing that a good minibatch scheme typically has relatively low correlation within each of the minibatches. Finally, we propose an accelerated primal-dual SGD algorithm in order to tackle another key bottleneck of stochastic optimization which is the heavy computation of proximal operators. The proposed method has fast convergence rate in practice, and is able to efficiently handle non-smooth regularization terms which are coupled with linear operators. 
\end{abstract}


\section{Introduction}

Stochastic gradient-based optimization algorithms have been ubiquitous in real world applications which involve solving large-scale and high-dimensional optimization tasks, particularly in the field of machine learning \cite{bottou2010large}, due to their scalability to the size of the optimization problems. In this work we study the practicality of stochastic gradient-based optimization algorithms in imaging inverse problems, which are also large-scale and high-dimension by nature. The class of problems we consider, with typical examples including image deblurring, denoising, inpainting, superresolution, demosaicing, tomographic image reconstruction, etc, can be generally formulated as the following:
\begin{equation}\label{obj}
   x^*  \in  \arg \min_{x \in \mathcal{X}} \left\{F(x) :=  \frac{1}{n}\sum_{i = 1}^n f_i(x) + \lambda g(x)\right\},
\end{equation}
where $\X \subseteq \mathbb{R}^d$ is a convex set and we denote by $f(x) = \frac{1}{n}\sum_{i = 1}^n f_i(x) := \frac{1}{n}\sum_{i = 1}^n \Bar{f}(a_i, b_i , x)$ the data fidelity term. We assume each $f_i(x) := \Bar{f}(a_i, b_i , x)$ to be proper, convex and smooth. In the classical setting of supervised machine learning, the variable $x$ contains the parameters of a classifier, while $A = [a_1; a_2; ...; a_n] \in \mb{R}^{n \times d}$ represents the features of training data samples, and $b = [b_1; b_2; ... ;b_n] \in \mb{R}^d$ denotes the corresponding labels. In the imaging inverse problems we are interested in this work, the variable $x$ represents the vectorized image, while $A = [a_1; a_2; ...; a_n]$ represent the forward model, and $b = [b_1; b_2; ... ;b_n]$ denotes the observations. To be more specific, we denote here a noisy linear measurement model with a ground-truth vectorized image $x^\dagger$ which is to be estimated, an $n$ by $d$ matrix $A$ which denotes the measurement operator, additive noise denoted by vector $w$, and the noisy measurement data denoted by vector $b \in \mb{R}^{n \times 1}$:
\begin{equation}\label{eq:1}
b = A x^\dagger + w, \ \ \ A \in \mb{R}^{n \times d}
\end{equation}
One of the most typical examples of the data fidelity term in imaging inverse problems is the least-squares loss:
\begin{equation}
    f(x) = \frac{1}{n} \sum_{i = 1}^n \frac{1}{2} (a_i^T x - b_i)^2= \frac{1}{2n}\|Ax - b\|_2^2,
\end{equation}
while we typically obtain a robust estimator of $x^\dagger$ via jointly minimizing the least-squares data-fidelity term with a structure-inducing regularization $g(\cdot)$ which encodes prior information we have regarding $x^\dagger$:
\begin{equation}\label{reg_ls}
    x^\star \in \arg \min_{x \in \X} \left\{ f(x) + \lambda g(x)\right\},\ \ \ f(x) := \frac{1}{2n}\|Ax - b\|_2^2
\end{equation}
The regularization term $g(x)$ is assumed to be a proper convex function and is possibly non-smooth. In imaging inverse problems, the most commonly used types of regularization are essentially sparsity-inducing norm penalty on either synthesis domain or analysis domain, with representative examples being the $\ell_1$ regularization on wavelet coefficients, and the total-variation (TV) regularization \cite{chambolle2016introduction}.

Traditionally, the imaging inverse problems are solved most often by minimizing the regularized least-squares via the deterministic first-order solvers, such as the proximal gradient descent \cite{lions1979splitting}, and its accelerated \cite{beck2009fast,liang2018improving} and primal-dual variants \cite{chambolle2011first,chambolle2016ergodic}. The iterates of the proximal gradient descent for solving (\ref{reg_ls}) can be written as:
 \begin{eqnarray*}
 && \mathrm{\textbf{Proximal gradient descent}} - \mathrm{Initialize}\ x^0 \in \X\\
 &&\mathrm{For} \ \ \ i = 0, 1, 2,...,  K\\
&&\left\lfloor
\begin{array}{l}
x^{i+1}=\mathrm{prox}_{\lambda g}^\eta [x^{i}-\eta \cdot \triangledown f(x^i)]
\end{array}
\right.
 \end{eqnarray*}
For least-squares data-fidelity term $\triangledown f(x^i) = \frac{1}{n}A^T(Ax^{i}-b)$. We denote the proximal operator as:
\begin{equation}
    \mathrm{prox}_{\lambda g}^\eta (\cdot) = \arg \min_{x \in \X} \frac{1}{2\eta} \|x - \cdot \|_2^2 + \lambda g(x).
\end{equation}

Stochastic gradient descent methods \cite{robbins1951stochastic,nemirovski1978cesari, nemirovsky1983problem,bottou2010large}, which is based on randomly selecting one or a few function $f_i(x)$ in each iteration and compute an efficient unbiased estimate of the full gradient $\triangledown f(x)$ and perform the descent step, by nature are the ideal solvers for the generic composite optimization task (\ref{obj}), including the regularized least-squares (\ref{reg_ls}). These type of methods are able to achieve scalablity to large-scale problems compared to the deterministic gradient methods in many modern machine learning applications.

In recent years, researchers have developed several advanced variants of stochastic gradient methods, namely, the variance-reduced stochastic gradient methods \cite{schmidt2013minimizing, johnson2013accelerating, defazio2014saga, xiao2014proximal}. In each iteration of these new stochastic algorithms, more delicated stochastic gradient estimator is computed, which can reduce the variance of the stochastic gradient estimator progressively, with small computational or storage overheads, and hence significantly improve the convergence rate of stochastic gradient methods. Most recently, with further combining the variance-reduction technique with the Nesterov's momentum acceleration technique which was originally designed to accelerate the deterministic gradient methods \cite{nesterov1983method, nesterov2007gradient}, researchers \cite{nitanda2014stochastic, lan2015optimal, allen2016katyusha, murata2017doubly} have developed several \say{optimal} stochastic gradient algorithms which can provably achieve the worse-case optimal convergence rate for (\ref{obj}).

While having been a proven success both in theory and in machine learning applications, there are few convincing results so far in the literature which report the performance of the stochastic gradient methods in imaging applications (with the possible exception of tomographic reconstruction \cite{karimi2016hybrid,karimi2017sparse,chambolle2018stochastic}). Can stochastic gradient methods significantly facilitate inverse problems as they did for machine learning? If not, why might stochastic optimization be inefficient for some inverse problems?  How can we understand such failures? How could we help practitioners to characterize whether a given inverse problem is suitable for stochastic gradient methods or not? This work is aimed at answering these questions in a systematic way.

\subsection{Highlights of this work}

We make the following contributions:

 \subsubsection{A metric for predicting stochastic acceleration}

 We first report surprisingly negative results of stochastic gradient methods in solving a space-varying image deblurring problem, which go against the conventional wisdom and common believe of the large-scale optimization and machine learning community. The first step of this work is to find out the key factor which determines the success or failure of stochastic gradient methods to be more advantageous than their deterministic counterparts for an imaging inverse problem. We start by a motivational analysis from known upper and lower complexity bounds for solving (\ref{obj}), demonstrating that in the worst case the acceleration given by stochastic gradient methods in terms of objective-gap convergence is dominated by this ratio. In the context of imaging inverse problem, it is more desirable to further study whether the acceleration provided by stochastic gradient methods in terms of estimation-error convergence is also dominated by this ratio. To show this, we provide a novel analysis for the estimation-error convergence rate of minibatch proximal SGD in solving linear inverse problems with regularization constraints, under expected smoothness \cite{gower2018stochastic} and restricted strong-convexity \cite{agarwal2012fast,2015_Oymak_Sharp} condition. By comparing our result for minibatch proximal SGD with the deterministic proximal gradient descent in the same setting, we can confirm that this ratio of Lipschitz constants is indeed the key factor which can be used to characterize whether a given inverse problem is suitable or not for applying stochatic gradient methods. Hence we find strong theoretical evidences, that the computational speedup which stochastic gradient methods can bring over their deterministic counterparts, is dominantly related to the ratio of the Lipschitz constants of the full gradient and the minibatch stochastic gradients.

 Based on our theoretical analysis, we propose to evaluate the limit of possible acceleration of a stochastic gradient method over its full gradient counterpart by measuring the \textit{Stochastic Acceleration} (SA) factors which are based on the ratio of the Lipschitz constants of the minibatched stochastic gradient and the full gradient. We also discover that the SA factors are able to characterize the benefits of using randomized optimization techniques, and that not all imaging problems have large SA factors. 
 
 \subsubsection{Understanding the relationship between the structure of inverse problems and stochastic acceleration} An immediate and crucial question to be answered is, \say{what types of inverse problems favor stochastic gradient algorithms?}. We provide tight and insightful lower and upper bounds for the stochastic acceleration factors for partition minibatch schemes, using standard tools in numerical linear algebra \cite{horn2012matrix}. Our lower bound results suggest that the SA factor we propose is directly related the ratio $\frac{\max_{i \in [n]} \|a_i\|_2^2}{\|A\|^2}$, which can be efficiently evaluated by practitioners. This ratio $\frac{\max_{i \in [n]} \|a_i\|_2^2}{\|A\|^2}$ is also directly related to the eigenspectrum of the Hessian matrix, when the measurements are relatively balanced, i.e. $\frac{\max \|a_i\|_2^2}{\frac{1}{n}\sum_{j \in [n]}\|a_j\|_2^2} = O(1)$, which is generally true for most of imaging inverse problems:

 \vspace{\baselineskip} 
 
 \noindent\fbox{%
    \parbox{.95\textwidth}{%
         {\it If such an inverse problem's Hessian matrix has a fast decaying eigenspectrum, then it can be guaranteed to have large SA factors, and hence can be characterized as a suitable application for stochastic gradient methods.}
    }%
}
 \vspace{\baselineskip}

And vice-versa: if such an inverse problems's Hessian matrix has a slowly-decaying eigenspectrum, then it is guaranteed to have small SA factors and can be deemed as unsuitable for stochastic gradient methods, not matter how delicately we partition the data.
 
While the spectral properties of the forward operator fundamentally controls the suitability of stochastic proximal gradient methods for an inverse problem, we know that for some inverse problems, different choices of partition can lead to different convergence rates for stochastic gradient algorithms in practice. One of our lower bounds for SA factors demonstrates that:

 \vspace{\baselineskip} 
 
 \noindent\fbox{%
    \parbox{.95\textwidth}{%
{\it If a partition scheme generates minibatches which have low local coherence structure, i.e. the measurements within minibatches are less correlated to each other, then it is superior to other partition schemes which have high local coherence structure}.}
}

 \vspace{\baselineskip} 
 
The SA factors and the lower bounds we propose provide for the practitioners efficient ways to check whether they should use stochastic proximal gradient techniques or classical deterministic proximal gradient methods to solve a given inverse problem, and also compare between different partition minibatch schemes and choose the best one among them in practice.

\subsubsection{Breaking the computational bottleneck of expensive/multiple proximal operators for momentum SGD} Another factor in imaging applications which significantly affects the SGD-type methods' actual performance is the frequent calculation of the costly proximal operators for the regularization terms, such as the TV semi-norm -- SGD methods need to calculate these much more frequently than full gradient methods. Moreover most of the fast SGD methods can not cope with more than one non-smooth regularization term \cite{zhao2018stochastic}. To overcome these issues we propose an accelerated primal-dual SGD (Acc-PD-SGD) algorithm based on the primal-dual hybrid gradient framework \cite{chambolle2011first,chambolle2016ergodic,zhao2018stochastic}, as a side-contribution. The proposed Acc-PD-SGD algorithm is able to efficiently handle (1) regularization with a linear operator, (2) multiple regularization terms,  while (3) maintaining Nesterov-type accelerated convergence speed in practice.

\subsection{Outline} Now we set out the rest of the paper. We start by presenting in section 2 a surprising negative result of state-of-the-art stochastic gradient methods in a space-varying image deblurring task. In section 3 we describe our notations and definitions which will be frequently used throughout the paper. Then in section 4, we provide theoretical analysis regarding the limitation of stochastic optimization algorithms, and particularly our novel analysis of minibatch SGD and the theory-inspired SA factors. In section 4, we also present bounds for the SA factors with respect to the spectral property of the forward operator, and hence derive a condition for an inverse problem to be a suitable application of stochastic gradient methods. In section 5 and 6, we present the accelerated primal-dual SGD algorithm and the numerical experiments. Final remarks appear in section 7, while we include the proofs of our theoretical results in the appendix.

\section{A motivating example}

	Image deblurring is an important type of imaging inverse problems and has been studied intensely during the recent decades.	For uniform deblurring, due to the cyclic structure of the deconvolution, FFT-based  ADMM\footnote{The computationally demanding sub-problems of {\it alternating direction method of multipliers} (ADMM) in this case can be solved with an efficient matrix inversion by FFT due to the cyclic structure of the uniform deconvolution.} variants have shown to be remarkably efficient \cite{afonso2010fast,almeida2013deconvolving, almeida2013blind} when compared to classic gradient-based solvers such as FISTA \cite{beck2009fast}. Such techniques, although being computationally efficient, are specifically tailored to a restricted range of problems where the observation models are diagonalizable by a DFT. For image deblurring, it is often not realistic to assume that imaging devices induce a uniform blur \cite{whyte2012non}. If the blurring is different across the image, then the efficient implementation of ADMM is not applicable in general. Then standard ADMM and deterministic gradient methods such as FISTA can be computationally expensive. It is therefore natural to ask: can stochastic gradient methods offer us a more efficient solution?

\begin{figure}[t] 
	\centering	
    \includegraphics[width=.85\textwidth]{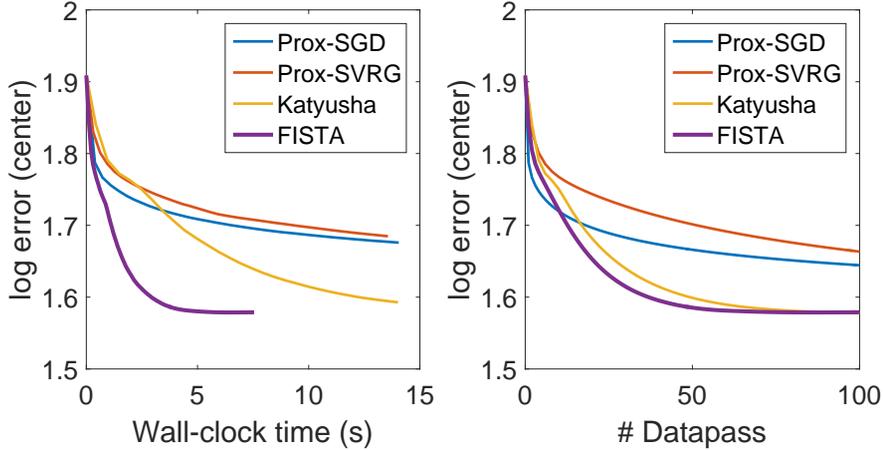}
	\caption{The estimation error plot for the deblurring experiment. The plots correspond to the estimation error of the central part (226 by 226) of the image.}
	\label{fig:1}
\end{figure}

\begin{figure}[t] 
	\centering	
    \includegraphics[width=.87\textwidth]{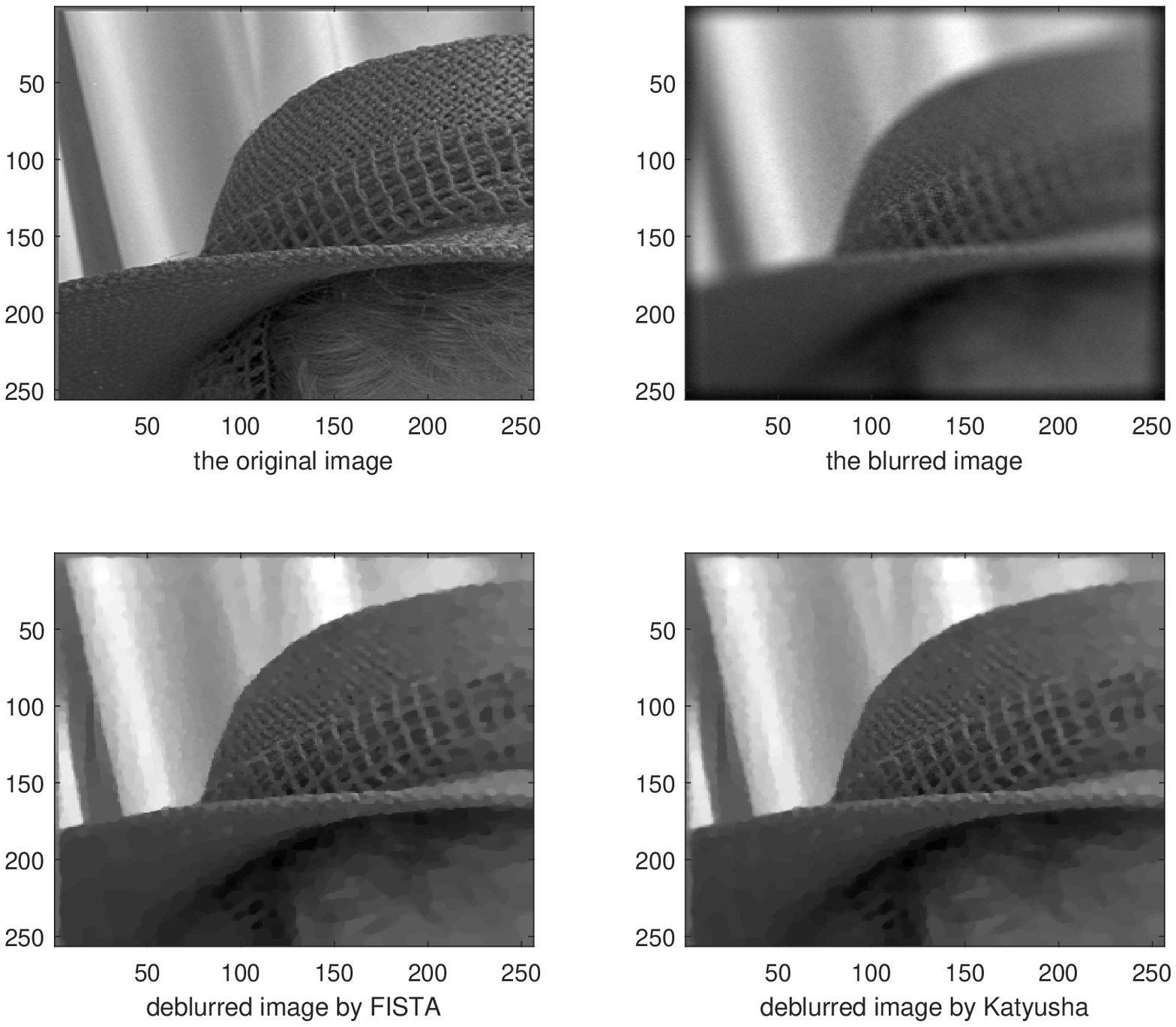}
	\caption{Up-left: the orignal image; up-right: the blurred image which is also corrupted with Gaussian noise; Down-left: deblurred image by FISTA; Down-right: deblurred image by Katyusha algorithm.}
	\label{fig:1_ima}
\end{figure}

We start by a simple space-varying deblurring \cite{whyte2012non} example where a part (sized 256 by 256) of the \say{Kodim04} image from \textit{Kodak Lossless True Color Image Suite} \cite{kodim} is blurred with a space-varying blur kernel which imposes less blurring at the center but increasingly severe blurring towards the edge. For the shape of the blur kernel, we choose the out-of-focus kernel provided in \cite{almeida2013deconvolving}. We also add a small amount of noise to the blurred image.

We test the effectiveness of several algorithms by solving the same TV-regularized least-squares problem, to get an estimation of the ground truth image. The algorithms we test in the experiments include the accelerated full gradient method FISTA \cite{beck2009fast}, proximal SGD \cite{rosasco2014convergence}, the proximal SVRG \cite{xiao2014proximal} and its accelerated variant, Katyusha algorithm \cite{allen2017katyusha} which has achieved optimal convergence rate in theory for (\ref{obj}). Perhaps surprisingly, on this experiment we report a negative result in Fig.\ref{fig:1} for all these randomized algorithms. The most efficient solver in this task is the full gradient method FISTA in terms of wall clock time and number of epochs (datapasses). The state-of-the-art stochastic gradient methods with Nesterov's acceleration even cannot beat FISTA in terms of epoch counts. For all the randomized algorithms we choose a minibatch size which is 10 percent of the total data size. For stochastic gradient methods, a smaller minibatch size in this case did not provide better performance in datapasses and significantly slowed down running time due to the multiple calls on the proximal operator.

\section{Notations and definitions}\label{notation}

We now make clear some notations which will occur frequently throughout this paper. We denote an image $X \in \mb{R}^{d_1 \times d_2}$ in its vectorized (raster) form $x \in \mb{R}^d$ where $d = d_1 \times d_2$, throughout this paper. Denote $X$'s columns as $x_1, x_2, ..., x_{d_2} \in \mb{R}^{d_1}$, and $X = [x_1, x_2, ..., x_{d_2}]$, then $x = [x_1; x_2; ...; x_{d_2}]$. Without specification, the scalar $n$ denotes the number of measurements, while $d$ denotes the dimension of measurements, and $m$ denotes the size of the minibatches, while $K$ is the number of minibatches. For a positive integer $q$, the notation $[q]$ represents the collection of all positive integers up to $q$ : $[1,...,q]$. When we write $m= \frac{n}{K}$, we implicitly assume that $n \mod K = 0$ -- this is just for simplification of presentation, without the loss of generality.

For a given vector $v$ and a scalar $p \geq 1$, we write its $\ell_p$ norm as $\|v\|_p$. We write the $j$-th row of $A$ as $a_j$, and $A = [a_1; a_2; ...; a_n]$. We denote the transpose of $A$ as $A^T$. We describe $\Bar{S} = [S_1, S_2, ..., S_K]$ as the partition of indices for a subsampling scheme, where $S_1 \cup S_2 \cup ... \cup S_K = [n]$ and $S_i \cap S_j = \emptyset, \forall i \neq j \in [n]$. Meanwhile,  we use superscript indexing $S^1, S^2,... , S^K$ to denote the corresponding row subsampling operators supported on the index set $S_1, S_2, ..., S_K$. For a given forward operator $A \in \mb{R}^{n \times d}$, we denote its spectral norm as $\|A\|$, and its Frobenius norm as $\|A\|_F$. We denote the $k$-th large eigenvalue of a symmetric matrix $H \in \mb{R}^{d \times d}$ as $\sigma(H, k)$. For $k > d$, we denote $\sigma(H, k) = 0$. We denote the $\ell_{1\rightarrow 2}$ inducing norm of $A$ as:
\begin{equation}
    \|A^T\|_{1 \rightarrow 2}^2 := \max_{i \in [n]}\|a_i\|_2^2.
\end{equation}
For an Euclidean vector space $\mc{X}$, we denote $\mc{F}_L^{p, q}(\mc{X})$ for the class of convex functions which are $p$-times differentiable while the $q$-th derivatives of them are $L$-Lipschitz continuous on $\mc{X}$. Without specification, we set $\mc{X} = \mb{R}^d$ throughout this paper. For a given minibatch index partition $[S_1, S_2, ... ,S_K]$ the minibatches and the gradients are defined as the following:
\begin{equation}
  f_{S_k} (x) = \frac{K}{n}\sum_{i \in S_k} f_i(x),\ \ \ \triangledown f_{S_k} (x) := \frac{K}{n}\sum_{i \in S_k} \triangledown f_i(x),\ \  k \in [K].
\end{equation}
The smoothness conditions of the full batch $f(x)$ and minibatches $f_{S_k} (x)$ are formally described as the following:
\begin{definition}
(Smoothness of the Full-Batch and the Mini-Batches.) $f(\cdot)$ is $L_f$-smooth and each $f_{S_k}(\cdot)$ is $L_{b}$-smooth, that is:
\begin{equation}
    f(x) - f(y) -\triangledown f(y)^T (x - y) \leq \frac{L_f}{2} \|x - y\|_2^2,\ \ \ \forall x, y \in \X,
\end{equation}
and
\begin{equation}\label{L_b_def}
    f_{S_k}(x) - f_{S_k}(y) -\triangledown f_{S_k}(y)^T (x - y) \leq \frac{L_b}{2} \|x - y\|_2^2,\forall x, y \in \X.
\end{equation}
\end{definition}
It is well known that, (\ref{L_b_def}) implies:
\begin{equation}
    \|\triangledown f_{S_k}(x) - \triangledown f_{S_k}(y) \|_2 \leq L_b \|x - y\|_2, \ \ \ \forall x, y \in \X,
\end{equation}
and,
\begin{equation}\label{L_b_def_2}
    \|\triangledown f_{S_k}(x) - \triangledown f_{S_k}(y) \|_2^2 \leq 2L_b [f_{S_k}(x) - f_{S_k}(y) - \langle \triangledown f_{S_k}(y), x - y \rangle]\ \ \ \forall x, y \in \X.
\end{equation}
We refer to \cite{bubeck2015convex} and \cite[Theorem 2.1.7]{nesterov2013introductory} for details.  In this paper we mainly consider two types of minibatch schemes, the partition minibatch sampling and random with-replacement sampling:
\begin{definition}[Partition minibatch sampling]\label{def_par_s}
 A subsampling matrix $S \in \mathbb{R}^{m \times n}$ is a partition minibatch sampling matrix if we pick it uniformly at random from a set of random subsampling matrix $[S^1; S^2; ...; S^K]$ where $S_1 \cup S_2 \cup ... \cup S_K = [n]$ and $S_i \cap S_j = \emptyset, \forall i \neq j \in [n]$.
\end{definition}
\begin{definition}[Random with-replacement sampling]\label{def_rand_s}
 A subsampling matrix $S \in \mathbb{R}^{m \times n}$ is an uniform random sampling matrix if we pick it uniformly at random from all possible $m$ row subset of $I_{n\times n}$.
\end{definition}

\section{Limitations of stochastic optimization}

The previous deblurring example appears to be contrary to the popular belief among the stochastic optimization community and the experience of machine learning practitioners, that stochastic gradient methods are much faster in terms of iteration complexity than deterministic gradient methods in solving large scale problems. To be specific -- to achieve an objective gap suboptimality of $F(x) - F(x^\star) \leq \varepsilon$, optimal stochastic gradient methods needs only $\Theta\left(n + \sqrt{n{L}/{\varepsilon}}\right)$ evaluations of $\triangledown f_i$ where $L$ denotes the gradient Lipschitz constant of $f_i$, see e.g. \cite{lan2015optimal,lin2015universal,allen2016katyusha}, while $\Theta\left(n\sqrt{{L}/{\varepsilon}}\right)$ are needed for optimal full gradient methods \cite{nesterov2007gradient}. Where is the loophole?

It is often easily ignored that the complexity results above are derived under different smoothness assumptions. For the convergence bound for full gradient, the full smooth part of the cost function $f(.)$ is assumed to be $L$-smooth, while for the case of stochastic gradient, every individual function $f_i(.)$ is assumed to be $L$-smooth. Now we can clearly see the subtlety: to compare these complexity results and make meaningful conclusions, one has to assume that these two Lipschitz constants are roughly the same. While this can be true, and is true for many problems, there are exceptions -- image deblurring is one of them.

For the case where the minibatch size is $1$, we can denote the smoothness constants of $f(\cdot)$ and $f_i(\cdot)$ as $L_f$ and $L_b$ respectively, we illustrate here some extreme examples for the two smoothness constants to demonstrate this possible dramatic difference:

Let $f(x) = \frac{1}{2n}\|Ax - b\|_2^2 = \frac{1}{n} \sum_{i = 1}^n \frac{1}{2} (a_i^T x - b_i)^2 := \frac{1}{n} \sum_{i = 1}^n f_i(x)$.

(1) If $a_1 = a_2 = a_3 ... , = a_{n-1} = a_n$, then $L_f = L_b$. 

(2) If $A = I$, then $L_f = \frac{1}{n} L_b$.

Now we turn to our analysis. Given a minibatch partition $[S_1, S_2, ... ,S_K]$ such that $S_1 \cup S_2 \cup...\cup S_K = [n]$ and:
\begin{equation}
  f_{S_k} (x) = \frac{K}{n}\sum_{i \in S_k} f_i(x),\ \ \ \triangledown f_{S_k} (x) := \frac{K}{n}\sum_{i \in S_k} \triangledown f_i(x), 
\end{equation}
while $k \in [K]$. In order to identify the potential of a certain optimization problem to be more efficiently solved using stochastic gradient methods, we start by deriving a motivating theorem comparing the convergence of the optimal full gradient methods as well as the optimal stochastic gradient methods.

\subsection{A motivational analysis}

 We consider comparing two classes of algorithm: the optimal deterministic gradient methods which meet the deterministic gradient-complexity lower bound \cite[Theorem 2.16]{nesterov2013introductory} presented in Theorem \ref{lower_bound_full_gradient} and the optimal stochastic gradient methods which are able to match the stochastic gradient-complexity lower bound \cite[Theorem 7]{woodworth2016tight} presented in Theorem \ref{lower_bound_SGD}. The FISTA algorithm and the Katyusha algorithm are typical instances from these two classes of algorithms. 
\begin{definition}\label{class_of_optimal_full}
(The class of optimal deterministic gradient algorithms.) A deterministic gradient method $\A_{\mr{full}}$ is called optimal if for any $s \geq 1$,  the update of $s$-th iteration $x^s_{\A_{\mr{full}}}$ satisfies:
\begin{equation}
    F(x^s_{\A_\mr{full}}) - F^\star \leq \frac{C_1 L_f \|x^0 - x^\star\|_2^2}{s^2},
\end{equation}
for some positive constant $C_1$.
\end{definition}
It is known that the FISTA algorithm satisfies this definition with $C_1 = 4$ \cite{beck2009fast}. We also define the class for optimal stochastic gradient methods:
\begin{definition}\label{class_of_stochastic}
(The class of optimal stochastic gradient algorithms.) A stochastic gradient method $\A_{\mr{stoc}}$ is called optimal if for any $s \geq 1$ and $K \geq 1$, after a number of $s\cdot K$ stochastic gradient evaluations, the output of the algorithm $x^s_{\A_{\mr{stoc}}}$ satisfies:
\begin{equation}
    \E F(x^s_{\A_{\mr{stoc}}}) - F^\star \leq \frac{C_2 [F(x^0) - F^\star]}{s^2} + \frac{C_3 L_b \|x^0 - x^\star\|_2^2}{Ks^2},
\end{equation}
for some positive constants $C_2$ and $C_3$.
\end{definition}
Note that the accelerated stochastic variance-reduced gradient methods such as Katyusha \cite{allen2016katyusha}, MiG\cite{zhou2018simple} and Point-SAGA \cite{defazio2016simple} satisfy this definition with different constants of $C_2$ and $C_3$.

Now we are ready to present the motivational theorem, which follows from simply combining the existing convergence results of the lower bounds for the stochastic and deterministic first-order optimization \cite{nesterov2013introductory,woodworth2016tight}.

\begin{theorem}\label{thm_6.3.1}
Let $g(.) = 0$, $m = K$. Denote an optimal deterministic algorithm ${\A_{\mr{full}}}$ which satisfies Def. \ref{class_of_optimal_full}, and an optimal stochastic gradient algorithm ${\A_{\mr{stoc}}}$ which satisifes Def. \ref{class_of_stochastic}. For a sufficiently large dimension $d$ and $\X = \left\{ x \in \mathbb{R}^d : \|x\|_2^2 \leq 1 \right\}$, there exists a set of convex and smooth functions $f_i \in \mathcal{F}_{L_b}^{1, 1}(\X)$, while $\frac{1}{K}\sum_{i=1}^K f_i = f \in \mathcal{F}_{L_f}^{1, 1}(\X)$, such that:
\begin{equation}
   \frac{\E f(x^s_{\A_{\mr{stoc}}}) - f^\star}{f(x^s_{\A_\mr{full}}) - f^\star} \geq   \frac{c_0 L_b}{KL_f} 
\end{equation}
for some positive constant $c_0$ which does not depend on $L_b$, $L_f$ and $K$. 
\end{theorem}

We provide the proof in Appendix \ref{Ad2}. From this theorem we can see that with the same epoch count, the ratio of the objective-gap sub-optimality achieved by $\A_{\mr{full}}$ and $\A_{\mr{stoc}}$ can be lower bounded by $\Omega(\frac{L_b}{KL_f})$ in the worst case. In other words, there exists a smooth finite-sum objective function, such that no optimal stochastic gradient method can achieve an acceleration on objective-gap more that $c_0 \cdot \frac{L_b}{KL_f}$ over any optimal deterministic gradient algorithm on minimizing this objective. Although the constant seems pessimistic, it is within our expectation since the lower bound on the convergence speed of stochastic gradient algorithms are derived on the worst possible function which satisfies the smoothness assumption. Motivated by the theory, we now further investigate and propose to evaluate the potential of stochastic acceleration simply by the ratio $\frac{KL_f}{L_b}$ which dominates our lower bound in Theorem \ref{thm_6.3.1}.

\subsection{An in-depth analysis of minibatch SGD for linear inverse problems}

In the previous subsection, we have provided a preliminary motivational analysis, which demonstrates that the speedup of stochastic gradient methods (with data-partition minibatches) over their deterministic counterparts in terms of objective gap convergence are at the worst case controlled by the ratio of Lipschitz constants of the stochastic gradient and full gradient, for the case of unregularized smooth optimization. Such analysis, although motivational, is restrictive in some aspects: in imaging inverse problems we usually consider non-smooth regularization, and we are more concerned with the convergence rates of optimization algorithms regarding estimation error. In this subsection, for the case where the linear measurements are noiseless (i.e. $\|w\|_2 = 0$), we provide a novel convergence rate analysis of minibatch SGD on solving constrained least-squares, which is a subclass of the regularized least-squares (\ref{reg_ls}). By comparing our rate of minibatch SGD with the best known result on deterministic proximal (projected) gradient descent (PGD) in the same setting, we confirm that the ratio of the Lipschitz constants of stochastic gradient and full gradient is indeed the key to characterize the practicality of stochastic optimization for a given inverse problem.

The constrained least-squares objective is written as the following:
\begin{equation}\label{eq:2}
    x^\star \in \arg \min_{x \in \mb{R}^d} \left\{ f(x) + \hat{g}(x)\right\},\ f(x) := \frac{1}{2n}\|Ax - b\|_2^2, \ \hat{g}(x) = \iota_\mathcal{K}(x),
\end{equation}
where the constraint set $\mc{K}$ is enforced as regularization, and the indicator function is used as the regularization to utilize the prior knowledge for better estimation:
\begin{equation}
 \iota_{\mathcal{K}}(x) =  \left\{
  \begin{tabular}{cc}
  $0$ &   \ if \ $x \in \mc{K}$. \\
  $+\infty$ & \ if \ $x \notin \mc{K}$.
  \end{tabular}
  \right.
\end{equation}
One typical example would be the total-variation (TV) semi-norm constraint in imaging applications such as inpainting and deblurring \cite{combettes2004image}, using an efficient TV-projection operator such as the one developed by \cite{fadili2010total}.

We restrict ourselves to make the convergence rate comparison of minibatch SGD and deterministic PGD on constrained least-squares mainly due to the fact that the restricted strong-convexity \cite{agarwal2010fast, agarwal2012fast, negahban2012unified, 2015_Oymak_Sharp}, which is essential for showing estimation-error convergence of the iterates, when applicable, is valid globally in this case since all descent directions are restricted within a tangent cone of the constraint set, as we will see. While for generic regularizers, such necessary restricted strong-convexity condition can only be valid locally \cite{agarwal2012fast}. Such an issue will make the desired accurate convergence rate comparison on the estimation error hopeless under the currently known framework for analyzing first-order methods, unless strong extra assumptions are made.

In this subsection, we also study the stochastic acceleration in the case where random with-replacement sampling is used instead of partitioning. Both random with-replacement minibatch scheme and the data-partitioning minibatch scheme are standard choices for stochastic gradient methods. Our analysis for minibatch SGD cover both data-partition minibatch schemes (Def. \ref{def_par_s}) and random with-replacement minibatch schemes (Def. \ref{def_rand_s}).

Unlike the analysis of data-partitioning sampling, a major difficulty for the analysis of random with-replacement sampling is that, the step-size choices suggested by the existing convergence results of minibatch proximal stochastic gradient methods can be highly suboptimal, which lead to conservative convergence rate guarantees. Fortunately, there has been recent progress \cite{pmlr-v97-qian19b,gazagnadou2019optimal} identifying near optimal step-size choices for minibatch stochastic gradient descent and SAGA algorithms for minimizing strongly convex and smooth objective functions. However, these existing results cannot be directly applied in inverse problems, mainly due to the following reasons: 

\begin{itemize}
    \item Firstly, these results are only for smooth optimization, while we often use non-smooth regularization in inverse problems, such as sparsity-inducing norms. It is unclear whether such large step-size choices are still allowed in the proximal setting.
    \item Secondly, these results require the objective function to be strongly-convex, which is not satisfied in general for inverse problems.
\end{itemize}

Due to these obstacles, the first step we should take is to extend the analysis of \cite{pmlr-v97-qian19b} to linear inverse problems with non-smooth regularization.

\subsubsection{Minibatch SGD for linear inverse problems with constraints}
\label{sec:main}

We denote $\mathcal{P}_\mathcal{K}$ as the orthogonal projection on to the set $\mathcal{K}$ and $\eta$ denotes the step size. We write down the minibatch SGD algorithm using with-replacement random sampling as the solver for (\ref{eq:2}):
 \begin{eqnarray*}
 && \mathrm{\textbf{Minibatch SGD}} - \mathrm{Initialize}\ x^0 \in \mathbb{R}^d\\
 &&\mathrm{For} \ \ \ i = 0, 1, 2,...,  K\\
&&\left\lfloor
\begin{array}{l}
x^{i+1}=\mathcal{P}_\mathcal{K}[x^{i}-\eta \triangledown f_{S_i}(x^i)]
\end{array}
\right.
 \end{eqnarray*}
 where $\triangledown f_{S_i}(x^i) = \frac{1}{m} (S^iA)^T(S^iAx^{i}-S^ib)$, and $S^i \in \mathbb{R}^{m \times n}$ are random subsampling matrices. It is well-known that, in order to ensure convergence, the random matrices $S^i$ need to satisfy:
\begin{equation}\label{ES}
    \E ({S^i}^T S^i) = \frac{m}{n} I, \forall i \in [K].
\end{equation}
In contrast to (\ref{L_b_def_2}), we introduce the notion of expected smoothness proposed by \cite{pmlr-v97-qian19b,gower2018stochastic}, which will be invloved in our analysis
\begin{definition}[Expected Smoothness]
 Let $\mc{D}$ be the distribution where the random subsampling operator $S$ is drawn from, and we denote $S_o$ as the corresponding index set subselected by $S$, the expected smoothness of the minibatches is defined as:
 \begin{equation}\label{es_def}
     \mb{E}_\mc{D} \|\triangledown f_{S_o}(x) - \triangledown f_{S_o}(y) \|_2^2 \leq 2 L_e [f(x) - f(y) - \langle \triangledown f(y), x - y \rangle],\ \  \forall x, y \in \mc{K}.
 \end{equation}
\end{definition}
If we use a data-partition minibatch, we have $L_e \leq L_b$ in (\ref{L_b_def_2}), as shown in \cite{pmlr-v97-qian19b}.

\subsubsection{Preliminaries for the analysis of minibatch SGD} We next provide a general theoretical framework for the analysis of minibatch SGD with the restricted strong convexity \cite{2015_Oymak_Sharp}.

\begin{definition} \label{D1}
 Cone $\mathcal{C}$ is the smallest cone at point $x^\dagger$ which contains the shifted set $\mathcal{K}-x^\dagger$:
 \begin{equation}
    \mathcal{C} := \left\{v \in \mathbb{R}^d |\  v = c(x - x^\dagger) , \forall c \geq 0, x \in \mathcal{K} \right\}.
\end{equation}
\end{definition}

\begin{definition}[Restricted Strong-Convexity]\label{D5}
 The restricted strong-convexity constant $\mathcal{\mu}_c$ is the largest non-negative value which satisfies: 
 \begin{equation}
    \frac{1}{n} \|Az\|_2^2 \geq \mu_c \|z\|_2^2 ,\ \ \forall z \in \mathcal{C}
 \end{equation}
\end{definition}
If the measurement system is noiseless, i.e. $\|w\|_2 = 0$, we expect the estimator (\ref{eq:2}) to be exact: $x^\star=x^\dagger$, if not, we expect the estimator to be robust to noise: $\|x^\star-x^\dagger\|_2\leq\frac{2\|w\|_2}{\epsilon}$. The success of exact/robust estimation is completely dependent on the null-space property of $A$ and the tangent cone $\mathcal{C} \in \mathcal{K}-x^\dagger$ on $x^\dagger$. In short, for the first scenario the necessary condition for exact recovery is $\|Az'\|_2 > 0$ for any normalized vector $z' \in \mathcal{C}$ \cite[Proposition 2.1]{2012_Chandrasekaran_Convex}; for the second scenario the necessary condition for robust recovery is $\|Az\|_2 \geq \epsilon \|z\|_2$  for any $z \in \mathcal{C}$ \cite[Proposition 2.2]{2012_Chandrasekaran_Convex}. This relationship between the null space property of $A$ and the constraint on $x^\dagger$ is fully captured by the restricted strong convexity property. If the restricted strong convexity condition is valid for (\ref{eq:2}), we know that $x^\star$ provides reliable and robust estimation for $x^\dagger$.

\subsubsection{Convergence result of minibatch SGD}
Using the expected smoothness result and restricted strong-convexity condition, we are able to derive the following convergence rate for minibatch proximal SGD under uniform random with-replacement minibatch scheme. If we set the step size of the minibatch SGD to be $\frac{1}{L_e}$, we can have the following linear convergence result up to a statistical accuracy:
\begin{theorem}[Convergence result for minibatch SGD]\label{T1}
 Suppose that $\|w\|_2 = 0$, let the step size of the minibatch SGD algorithm $\eta = \frac{1}{L_e}$. The expected error of the update by the $i$-th iteration obeys:
 \begin{equation}
     \E(\|x^i-x^\dagger\|_2) \leq \left(1 - \frac{\mu_c}{ L_e}\right)^{\frac{i}{2}}\|x^0-x^\dagger\|_2.
 \end{equation}
\end{theorem}
We include the proof of this convergence theorem in Appendix \ref{Ad3}. The convergence result of the deterministic projected gradient descent under restricted strong-convexity is well-studied in the literature, and we present it here for comparison, while the proof is simple, following a similar procedure to that in, e.g. \cite{2015_Oymak_Sharp, tang2017exploiting}:
\begin{theorem}[Convergence result for deterministic projected gradient descent]\label{T2}
 Suppose that $\|w\|_2 = 0$, let the step-size of projected gradient descent algorithm be $\eta = \frac{1}{L_f}$, the estimation error of the update by the $i$-th iteration obeys:
\begin{equation}
    \|x^i - x^\dagger\|_2 \leq \left( 1 - \frac{\mu_c}{L_f} \right)^i \|x^0 - x^\dagger\|_2.
\end{equation}
\end{theorem}

\begin{remark}\label{remark_1}

We can compare our convergence result of minibatch SGD in Theorem \ref{T1} with the result for deterministic projected gradient descent in Theorem \ref{T2}. To guarantee an estimation accuracy $\|x^N - x^\dagger \|_2 \leq \varepsilon$, the deterministic proximal gradient descent needs:
 \begin{equation}
     N_{\mr{full}} = O\left( \frac{L_f}{\mu_c} \log\frac{1}{\varepsilon} \right)
 \end{equation}
iterations, while the minibatch SGD needs:
 \begin{equation}
     N_{\mr{stoc}} = O\left( \frac{L_e}{\mu_c} \log\frac{1}{\varepsilon} \right).
 \end{equation}
Hence the iteration complexity of minibatch SGD is $O(\Upsilon_e)$-times smaller than the projected gradient descent, where:
\begin{equation}
    \Upsilon_e = \frac{\frac{n}{m} N_{\mr{full}}}{N_{\mr{stoc}}} = \frac{\frac{n}{m}L_f}{L_e} 
\end{equation}
For the data-partition minibatch scheme where we partition the least-squares loss function in to $K = \frac{n}{m} $ minibatches, we know that $L_e \leq L_b$, as shown in \cite[proposition 3.7]{pmlr-v97-qian19b}. Hence we have:
\begin{equation}
    \Upsilon_e = \frac{\frac{n}{m}L_f}{L_e} \geq \frac{K L_f}{L_b},
\end{equation}
which demonstrates that for data-partition minibatch scheme, the acceleration of minibatch SGD can offer over its deterministic counterpart, is dominated by the ratio of the Lipschitz constants of the full gradient and minibatch stochastic gradient.
\end{remark}

\subsection{Evaluating the limitation of SGD-type algorithms}

We introduce a metric called the {\it Stochastic Acceleration} (SA) factor based on our theoretical analysis of minibatch SGD in the previous section. The curve for SA factor as a function of the minibatch number $K$ (for a given minibatch pattern) is able to provide a way of characterizing inherently whether for a given inverse problem and a certain partition minibatch sampling scheme,  randomized gradient methods should be preferred over the deterministic full gradient methods or not.

\begin{definition}\label{SA_1_def}

 For a given disjoint partition minibatch index $[n] = S_1 \cup S_2 \cup ... \cup S_{K} := \Bar{S}$ where $S_i \cap S_j = \emptyset, \forall i \neq j \in [K]$, with corresponding subsampling operators $[S^1, ... S^K]$, the Stochastic Acceleration (SA) factor is defined as:
\begin{equation}
   \Upsilon(A , \Bar{S}, K) = \frac{KL_f}{L_b}.
\end{equation}
\end{definition}


We next evaluate the SA factors for the least squares loss function $f(x) = \frac{1}{2n} \|Ax - b\|_2^2$ with different types of forward operator. We use the without-replacement sampling (data-partitioning) which are most applied in practice. In this case we have
\begin{equation}\label{ls1}
      f(x) = \frac{1}{2n}\|Ax - b\|_2^2 = \frac{1}{K} \sum_{k = 1}^K f_{S_k} (x) ,
\end{equation}
where,
\begin{equation}\label{lss1}
     f_{S_k} (x) := \frac{K}{2n}\|S^k A x - S^k b\|_2^2,
\end{equation}
The examples of forward operator $A$ we consider here include the space-varying deblurring ($A_\mr{blur} \in \mb{R}^{262144 \times 262144}$), a random compressed sensing matrix with i.i.d Guassian random entries (with a size $A_\mr{rand} \in \mb{R}^{500\times 2000}$), a fan beam X-ray CT operator ($A_\mr{CT} \in \mb{R}^{91240 \times 65536}$). Meanwhile, in order to contrast with the application of stochastic gradient algorithms in machine learning, we also consider linear regression problems on two machine learning datasets: RCV1 dataset ($A_\mr{rcv1} \in \mb{R}^{20242 \times 47236}$), and Magic04 ($A_\mr{magic04} \in \mb{R}^{19000 \times 50}$). 

For the X-ray CT image reconstruction example and deblurring example we use TV regularization for $g(x)$ in (\ref{obj}), while for the rest of the examples we use $\ell_1$ regularization. We vectorize the image precisely as described in section \ref{notation}. The data-partition we choose is the interleaved sampling for all the examples, where the $k$-th minibatch is formed as the following:
\begin{eqnarray}\label{interleave}
    f_{S_k} (x) &:=& \frac{K}{n} \sum_{i = 1}^{\lfloor{n/K}\rfloor} f_{k + iK}(x) \\&=& \frac{K}{2n} \sum_{i = 1}^{\lfloor{n/K}\rfloor} (a_{k + iK}^T x - b_{k + iK})
\end{eqnarray}
In Figure \ref{EA_fig}, we demonstrate the SA factors for these 5 problem instances as a function of the number of minibatches along with the empirical acceleration observed when solving these problems. From the result demonstrated in the Figure \ref{EA_fig} we find that indeed the stochastic methods have a limitation on some optimization problems like deblurring and inverse problems with random matrices, where we see that the curve for the SA factor of such problems stays low and flat even when we increase the number of minibatches. For the machine learning datasets and X-ray CT imaging, the SA factor increases rapidly and almost linearly as we increase the number of minibatches, which is in line with observations in machine learning on the superiority of SGD and also the observation in CT image reconstruction of the benefits of using the ordered-subset methods \cite{erdogan1999ordered,kim2015combining} which are similar to stochastic gradient methods.

\begin{figure}[t] 
	\centering	
   { \includegraphics[width=.495\textwidth]{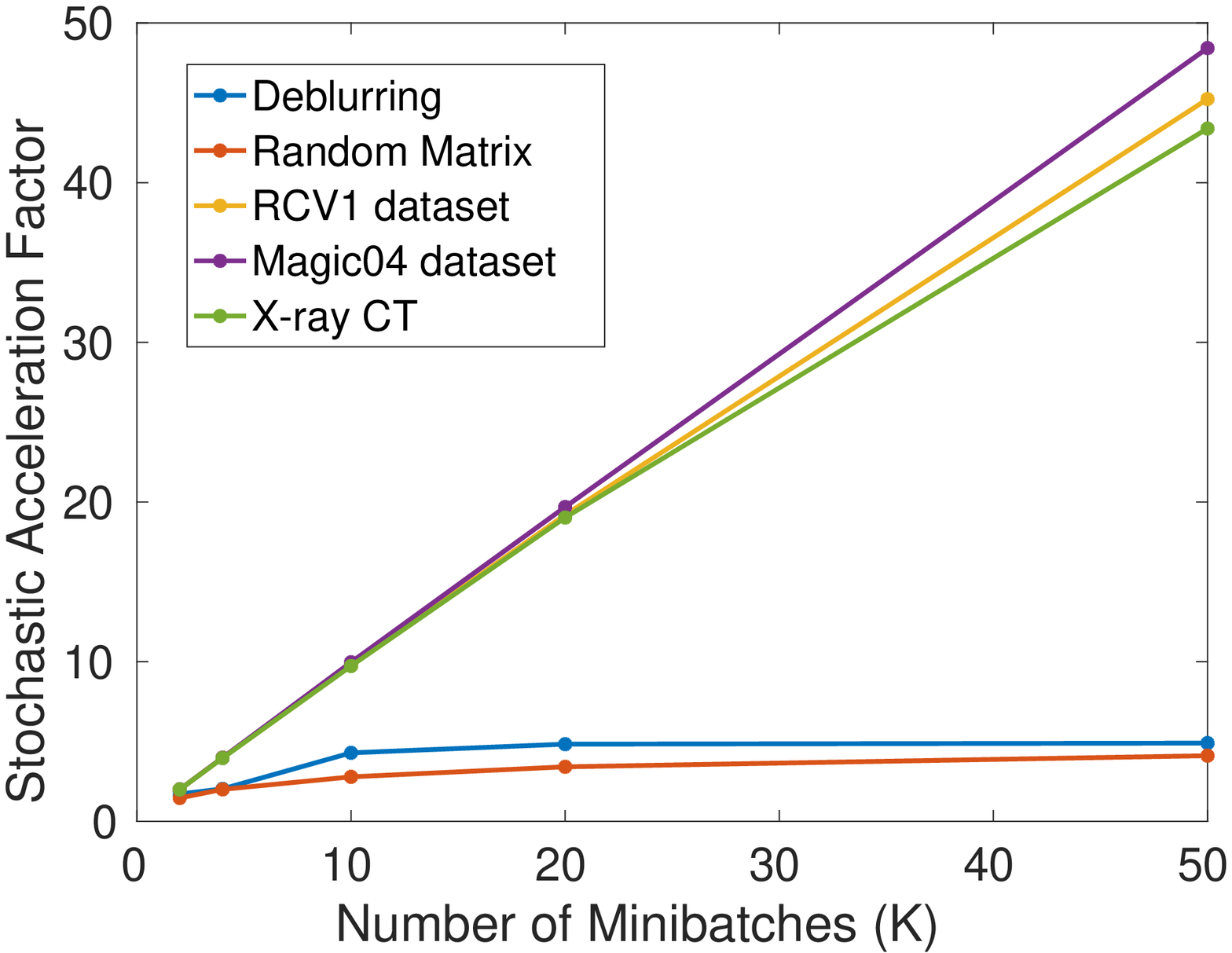}}
    {\includegraphics[width=.495\textwidth]{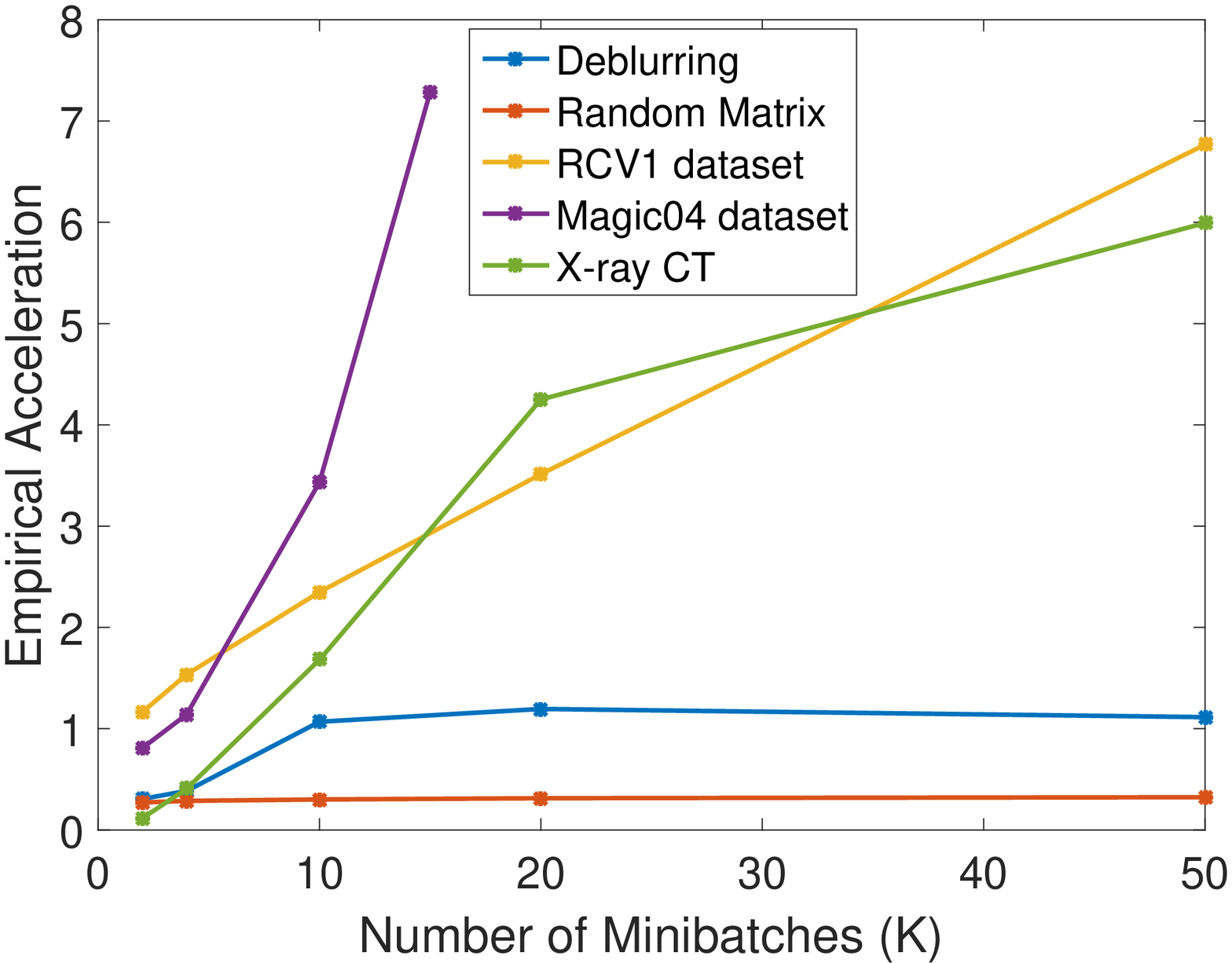}}
	\caption{  (a) Stochastic Acceleration (SA) function of inverse problems with different forward operators. (b) Empirical observation comparing the objective gap convergence of Katyusha and FISTA algorithm in 15 epochs.}
	\label{EA_fig}
\end{figure}

The curves for the SA factor on the first figure qualitatively predict the empirical comparison result\footnote{We compare the objective-gap convergence of FISTA and Katyusha for a fixed number of datapasses (epochs).} of the Katyusha and FISTA algorithms shown on the second, where we observe that Katyusha offers no acceleration over the FISTA on either the deblurring or the Gaussian random inverse problem, but significantly outperforms FISTA on the other cases. Indeed, positive results for applying SGD-type algorithms on these problems are well-known already \cite{bottou2010large, xiao2014proximal,erdogan1999ordered}. Hence we have shown that the SA factor we propose is useful in characterizing whether an inverse problem is inherently a suitable candidate for stochastic gradient methods for a given partition.

\subsection{Local coherence structure, eigenspectrum, and stochastic acceleration}

From the results we have obtained so far, we now go deeper to investigate the relationship of the SA factor and the structure of the forward operator of the inverse problem.

Subsequently we will assume that each partition has an equal size $m$ for the simplicity of presentation. We will find the following definition of the {\it local-accumulated-coherence} to be useful.
\begin{definition}[Local-Accumulated-Coherence]
 Give a partition $[n] = S_1 \cup S_2 \cup ... \cup S_{K} := \Bar{S}$ for $A = [a_1;a_2;...;a_n]$ which 
\begin{equation}\label{local_coherence}
    \mu_\ell(A, \Bar{S}, K) = \max_{q \in [K]} \max_{j \in S_q} \sum_{k \in S_q} |\langle a_j, a_k \rangle|.
\end{equation}
\end{definition}
Our definition of the local cumulative coherence is similar to and more general than the one presented in \cite{tropp2004greed} and related works, but we do not require the rows to be normalized and the summation includes the term $|\langle a_j, a_j \rangle|$. The local-accumulated-coherence captures the the correlation characteristic between the linear measurements within each partitioned minibatches. As we will see, if a partition have a smaller local accumulated coherence than another partition, then it typically can have a better SA factor. Our analysis in this subsection is based on the Gersgorin disk theorem:
\begin{theorem}[Gersgorin disk theorem \cite{horn2012matrix}]
Every eigenvalue of a Hermitian matrix $H \in \mb{R}^{n \times n}$ lives in a union of disks:
\begin{equation}
    G(H) = \bigcup_{i = 1}^n \left\{ x \in \mb{R} : |x - H_{ii}| \leq \sum_{j \neq i} |H_{ij}| \right\}
\end{equation}
\end{theorem}
The Gersgorin disk theorem relates a square symmetric matrix's eigenvalues with its entries, and links to the gradient-Lipschitz constant $L_b$ -- which in the least-squares context can be written as:
\begin{equation}
    L_b = \frac{K}{n} \max_{k \in [K]}\|S^kA (S^k A)^T\|
\end{equation}
in linear inverse problems. With this we can lower-bound the SA factors with the ratio of $\frac{\|A\|^2}{\mu_\ell}$.

\begin{theorem}[Lower bounds for $\Upsilon(A, \Bar{S}, K)$]\label{Thm1_lower_bound}
The SA factor for any linear inverse problem with $f(x) = \frac{1}{2n}\|Ax - b\|_2^2$ is lower bounded as:
\begin{eqnarray}\label{lower_bound_SA}
    \Upsilon(A, \Bar{S}, K) &\geq& \alpha_\ell(A, \Bar{S}, K) := \frac{\|A\|^2}{\mu_\ell(A, \Bar{S}, K)} \\&\geq& \alpha_u(A, K) := \frac{K\|A\|^2}{n \|A^T\|_{1 \rightarrow 2}^2}\\ &\geq& \alpha_s(A, K) := \frac{K\cdot\sigma(A^TA, 1)}{\rho \cdot \sum_{i = 1}^d \sigma(A^TA, i)},
\end{eqnarray}
where $\rho \in [1, n]$ satisfies:
\begin{equation}\label{ass_b}
    \frac{\max_{i \in [n]} \|a_i\|_2^2}{\frac{1}{n}\sum_{j = 1}^n\|a_j\|_2^2} \leq \rho.
\end{equation}
\end{theorem}

We provide the proof in Appendix \ref{Ad5}. Note that most inverse problems we encounter usually satisfy (\ref{ass_b}) with $\rho = O(1)$, since unlike machine learning applications where we may often have outliers in datasets, most inverse problems are well-designed with measurements having similar amount of energy. The second and the third inequalities are partition-independent and reveal a strong relationship between the SA factor and the spectral properties of the forward operator. The third inequality in (\ref{lower_bound_SA}) clearly states that, {\it if a linear inverse problem which satisfies (\ref{ass_b}) with $\rho = O(1)$ has a Hessian with a fast-decaying eigenspectrum, it can be guaranteed to have good SA factors}. They are also tight bounds if we do not impose additional structural assumptions on the forward operator -- if $A$ has identical rows we have $\Upsilon(A, \Bar{S}, K) = \frac{K\|A\|^2}{n \|A^T\|_{1 \rightarrow 2}^2} = K$,  noting that $\|A^T\|_{1 \rightarrow 2}^2 = \max_{i \in [n]} \|a_i\|_2^2$. If $\rho = 1$ which means that rows of $A$ have the same $\ell_2$ norm, we have $\alpha_u(A, K) = \alpha_s(A, K)$.

 We have derived partition independent lower bounds for $\Upsilon(A, \Bar{S}, K)$. However, these lower bounds, by definition have to cover the worst case of partition. Hence for some inverse problems which admit inferior partitions, these may be crude lower bounds. It is therefore insightful to derive a lower bound for the case where we randomly partition the data. We provide the following lower bound using the Matrix Chernoff inequality and the union bound, following a similar argument by \cite[Proposition 3.3]{needell2014paved}. We present the proof in Appendix \ref{Ad7}.
 
 \begin{theorem}[Lower bounds for $\Upsilon(A, \Bar{S}, K)$ for a random partition]\label{Thm_rand par}
  If $\Bar{S}$ is a uniform random partition, then for $K \in \left[ \frac{\|A\|^2}{\|A^T\|_{1 \rightarrow 2}^2}, \min(n,d)\right]$, the following lower bounds hold:
  \begin{eqnarray}\label{lower_bound_rand_par}
      \Upsilon(A, \Bar{S}, K) &\geq& \alpha_r(A, K, \delta) := \frac{1}{\frac{1}{K} + \delta \cdot \frac{\|A^T\|_{1 \rightarrow 2}^2}{\|A\|^2}}\\ &\geq& \alpha_\sigma(A, K, \delta) := \frac{1}{\frac{1}{K} + \delta \cdot \frac{\rho}{n}\cdot \frac{\sum_{i = 1}^d \sigma(A^TA, i)}{\sigma(A^TA, 1)}} 
  \end{eqnarray}
  with probability at least: $1 - d^2 \left( \frac{e}{\delta} \right)^{\delta }$.
 \end{theorem}
 
 \begin{remark}
 This lower bound for random partition scheme again demonstrates the strong relationship between SA factor and the ratio $\frac{\|A^T\|_{1 \rightarrow 2}^2}{\|A\|^2}$ which is controlled by the eigenspectrum $\frac{\sum_{i = 1}^d \sigma(A^TA, i)}{\sigma(A^TA, 1)}$. Note that due to the Matrix Chernoff inequality \cite{tropp2012user}, this theorem holds with a probability $1 - d^2 \left( \frac{e}{\delta} \right)^{\delta }$, which is dimension-dependent, and meanwhile we also need to note that the theorem covers only the regime where $K$ is sufficiently large. For a fixed dimension $d$, the smaller the ratio $\frac{\|A^T\|_{1 \rightarrow 2}^2}{\|A\|^2}$, or rather, the faster the eigenspectrum decays, the larger SA factors will be for the number of minibatch $K$ within the range $\left[ \frac{\|A\|^2}{\|A^T\|_{1 \rightarrow 2}^2}, \min(n,d)\right]$.

 To be more specific, if we demand here $1 - d^2 \left( \frac{e}{\delta} \right)^{\delta } \geq 0.9$, we can compute for a given $\delta$ the maximum allowed $d$ for the Theorem \ref{Thm_rand par} to hold with this probability via the following bound:
 \begin{equation}
     d \leq \sqrt{\frac{1}{10 \left( \frac{e}{\delta} \right)^{\delta }}}.
 \end{equation}
 We list the values of $\sqrt{\frac{1}{10 \left( \frac{e}{\delta} \right)^{\delta }}}$ for a range of $\delta$ in table \ref{table_d}.
 \begin{table}[h]
\caption{Maximum allowed $d$ for Theorem \ref{Thm_rand par} to hold with probability at least 0.9}
\label{table_d}
\vskip 0.15in
\begin{center}
\begin{small}
\begin{sc}
\begin{tabular}{lccccccr}
\hline
$\delta$ & 15 & 17 & 19 & 21 & 23 & 25\\
\hline
Max. $d$   & $1.16 \times 10^5$  &  $1.85 \times 10^6$ & $3.32 \times 10^7$ & $6.65 \times 10^8$ & $1.46 \times 10^{10}$ & $3.51 \times 10^{11}$\\
\hline
\end{tabular}
\end{sc}
\end{small}
\end{center}
\vskip -0.1in
\end{table}

 Moreover, as we will show, qualitatively this works but using a smaller $\delta$ seems to better describe what we see in the SA factor for random partition. On the other hand, this result requires the number of minibatches $K$ to be sufficiently large ($K \in \left[ \frac{\|A\|^2}{\|A^T\|_{1 \rightarrow 2}^2}, \min(n,d)\right]$), however in our numerical experiments we can observe that for small values of $K$ the lower bound $\alpha_r(A, K, \delta)$ still provides a reasonably good estimate. Similar restrictions occur in \cite[Proposition 3.3]{needell2014paved} which is also based on the Matrix Chernoff inequality. Whether such a restriction can be technically removed is an interesting open question.
 \end{remark}

Meanwhile, we can also have an upper bound for the SA factor, independent of the partition $\Bar{S}$, in terms of the eigenspectrum of the Hessian matrix $A^TA$. This upper bound can be derived from a standard result \cite[Theorem 4.3.15]{horn2012matrix} using the fact that the matrix $(S^k A)^T S^k A$ and $S^k A(S^k A)^T$ share the same non-zero eigenvalues. We denote the $i$-th large eigenvalue of a Hermitian matrix $H$ as $\sigma(H, i)$, and the upper bound is written as the following:
\begin{theorem}[Upper bound for $\Upsilon(A, \Bar{S}, K)$]\label{Thm_upper_bound}
The SA factor for any linear inverse problem with $f(x) = \frac{1}{2n}\|Ax - b\|_2^2$ is upper bounded as:
 \begin{equation}\label{upp_b}
    \Upsilon(A, \Bar{S}, K) \leq \beta(A, K) := \frac{\sigma\left(A^TA, 1\right)}{\sigma\left(A^TA, {\lfloor{n - \frac{n}{K} + 1}\rfloor}\right)},
\end{equation}
for any possible partition $\Bar{S}$.
\end{theorem}
We include the proof in Appendix \ref{Ad6} for completeness. The upper bound (\ref{upp_b}) suggests that, if the Hessian matrix $A^TA$ has slowly-decaying eigenvalues at the tail, it indeed typically cannot have a large SA factor, no-matter how delicately we partition the forward operator $A$. The upper bound and the lower-bounds jointly suggest that, having a fast-decaying eigenspectrum of the Hessian is a sufficient and necessary condition for an inverse problem to have good SA factors.

\subsubsection{Numerical examples}\label{discussion_bounds}

\begin{figure}[t] 
	\centering	
  {  \includegraphics[width=.47\textwidth]{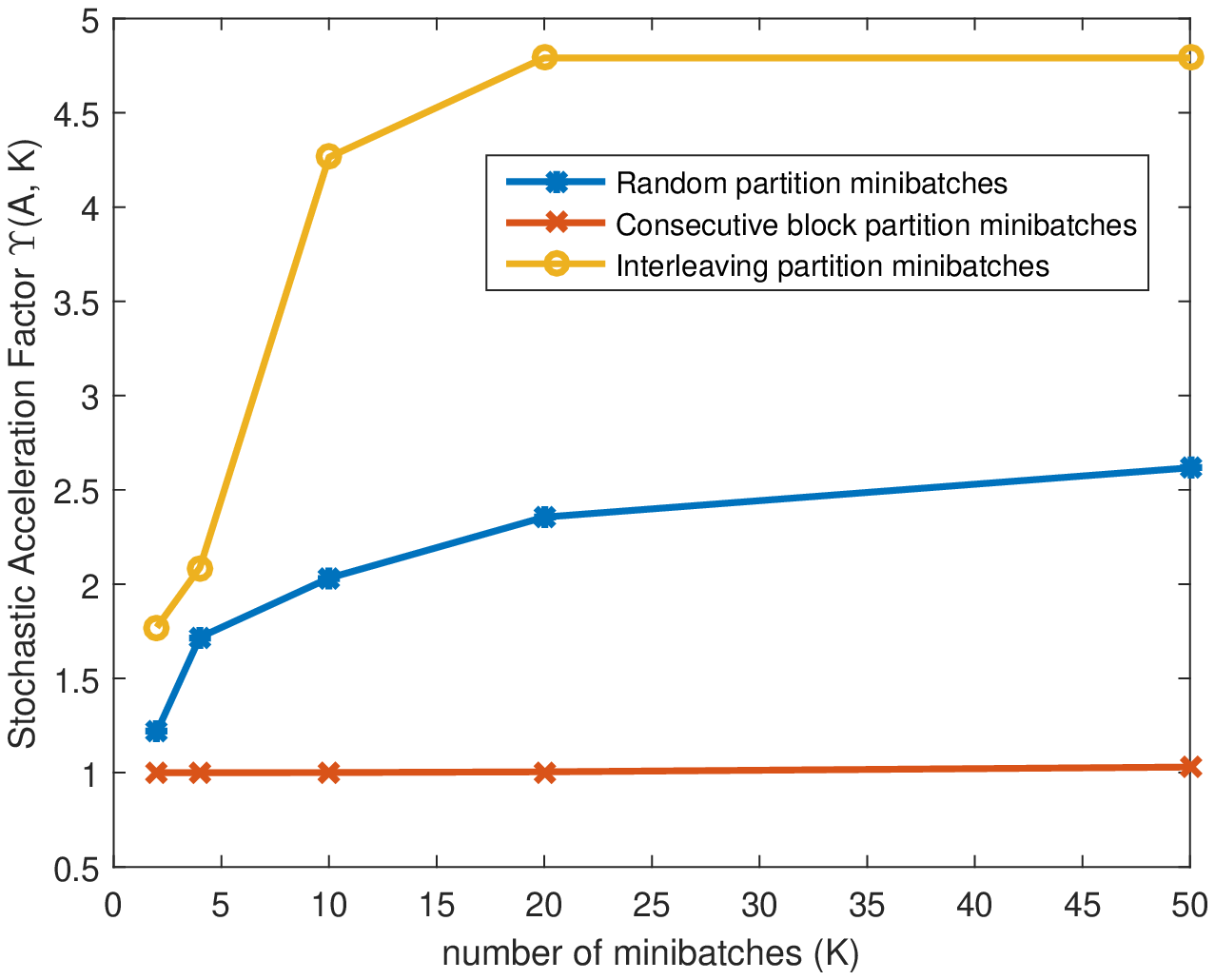}}
   { \includegraphics[width=.47\textwidth]{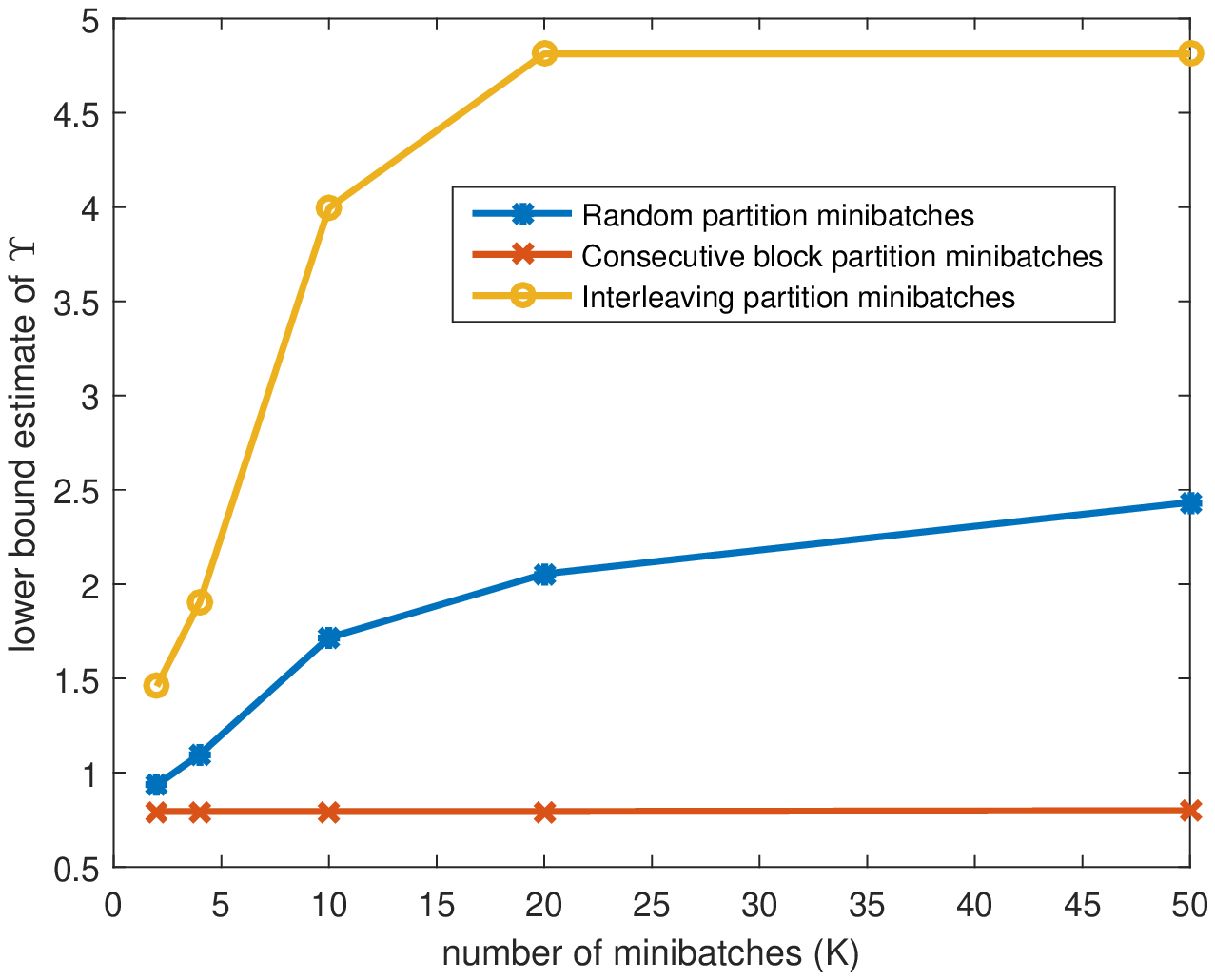}}  \\     {\includegraphics[width=.47\textwidth]{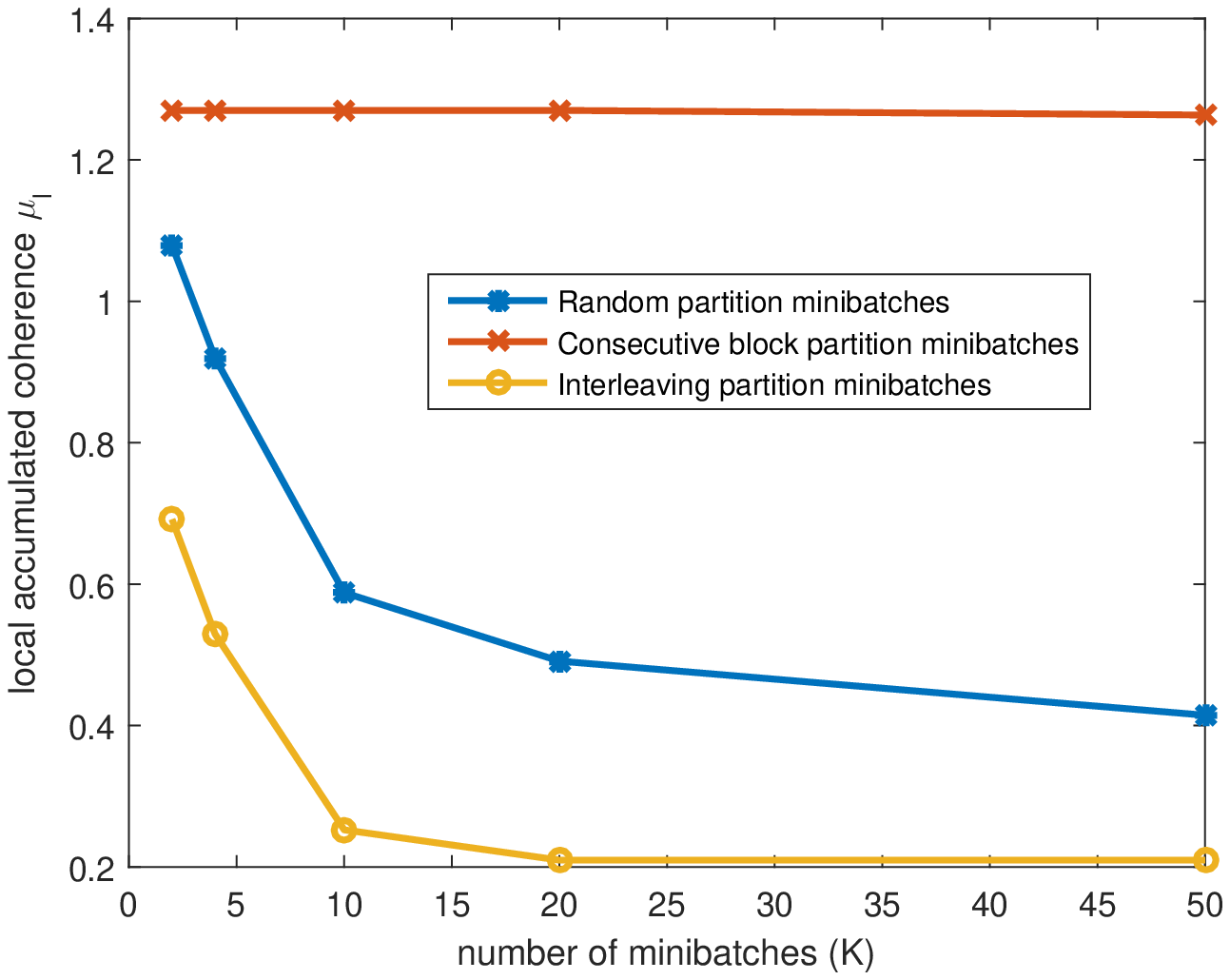}}
	\caption{  (a) Stochastic Acceleration (SA) function under different partition choices for space-varying deblurring task.  (b) lower bound estimate of SA factors via local accumulated coherence. (c) local accumulated coherence for each partition scheme.}
	\label{EA_2f}
\end{figure}

A key result in this analysis is the lower bound which relates the stochastic acceleration factor with the local coherence/correlation structure of the given partition:
\begin{equation}
    \Upsilon(A, \Bar{S}, K) := \frac{KL_f}{L_b} \geq \frac{\|A\|^2}{\mu_\ell(A, \Bar{S}, K)}.
\end{equation}
An immediate conclusion we can have is that, for a given linear inverse problem with forward operator $A$, and a partition $\Bar{S}$, {\it the smaller the local accumulated coherence is, the larger the SA factor will be}. For some inverse problems, judiciously choosing the partition for minibatches is important -- good choices of partitioning can have small local coherence and hence lead to larger SA factors in practice. 

\begin{figure}[t] 
	\centering	
  {\includegraphics[width=.38\textwidth]{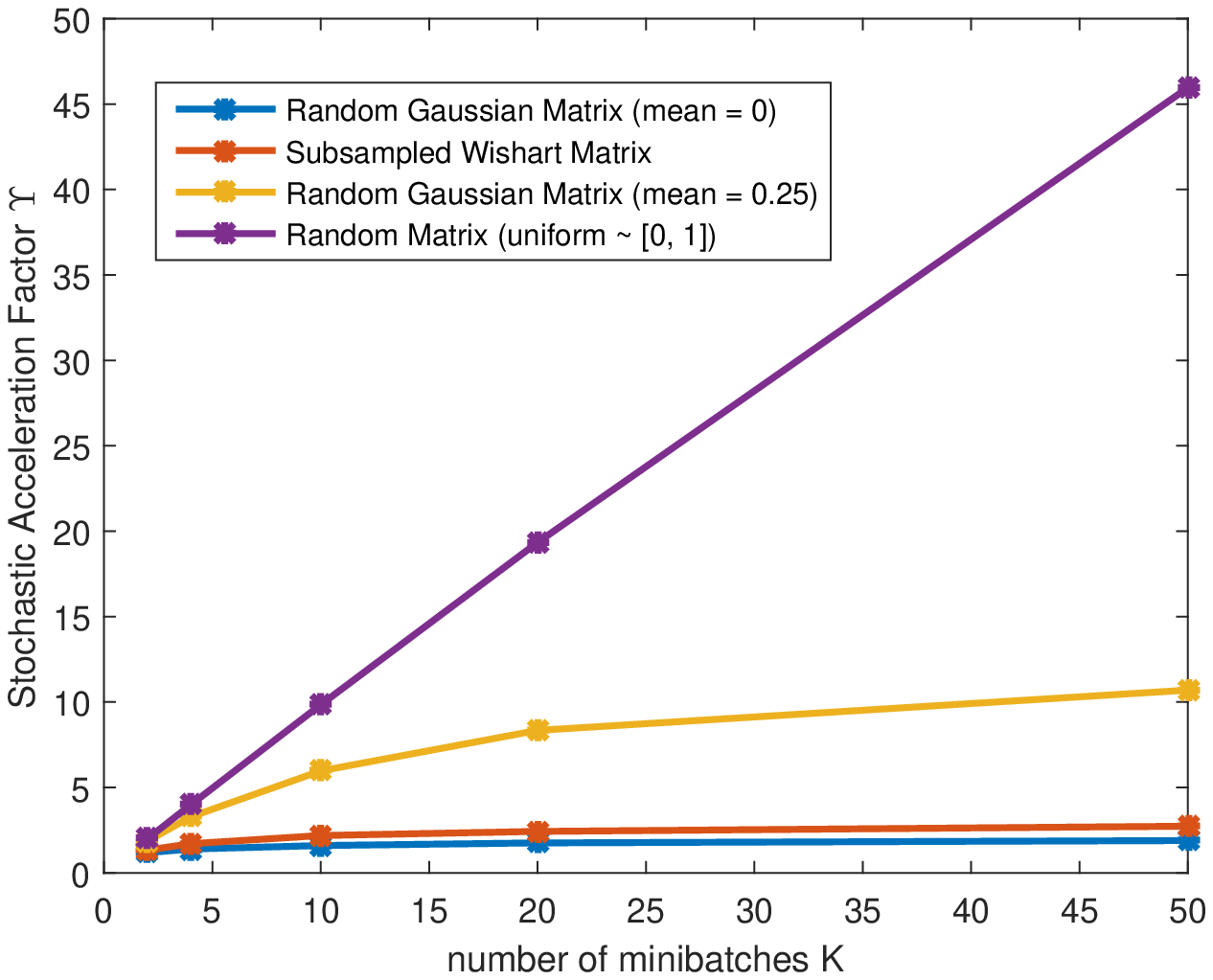}} 
  {\includegraphics[width=.38\textwidth]{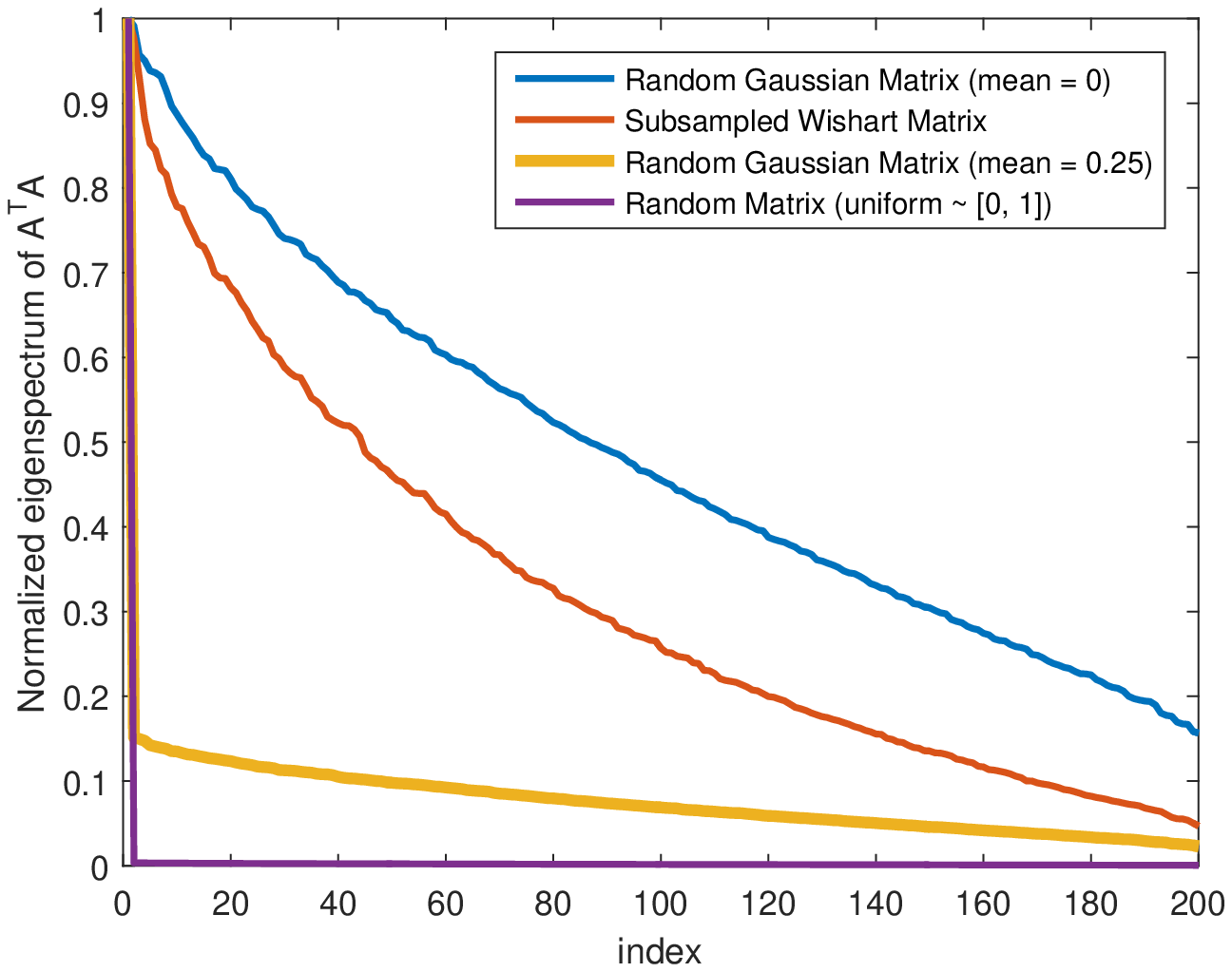}} \\   
  {\includegraphics[width=.38\textwidth]{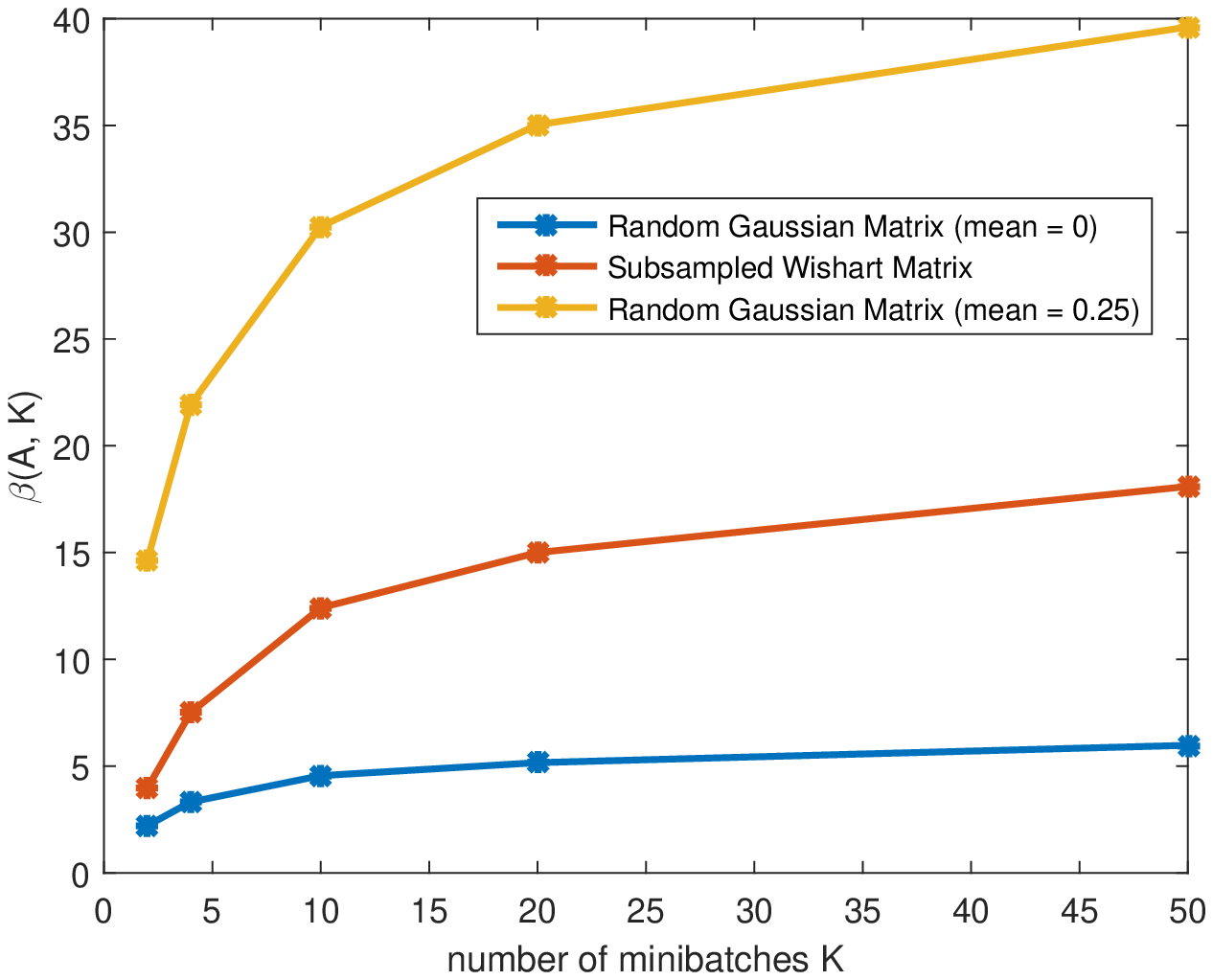}}
     {\includegraphics[width=.38\textwidth]{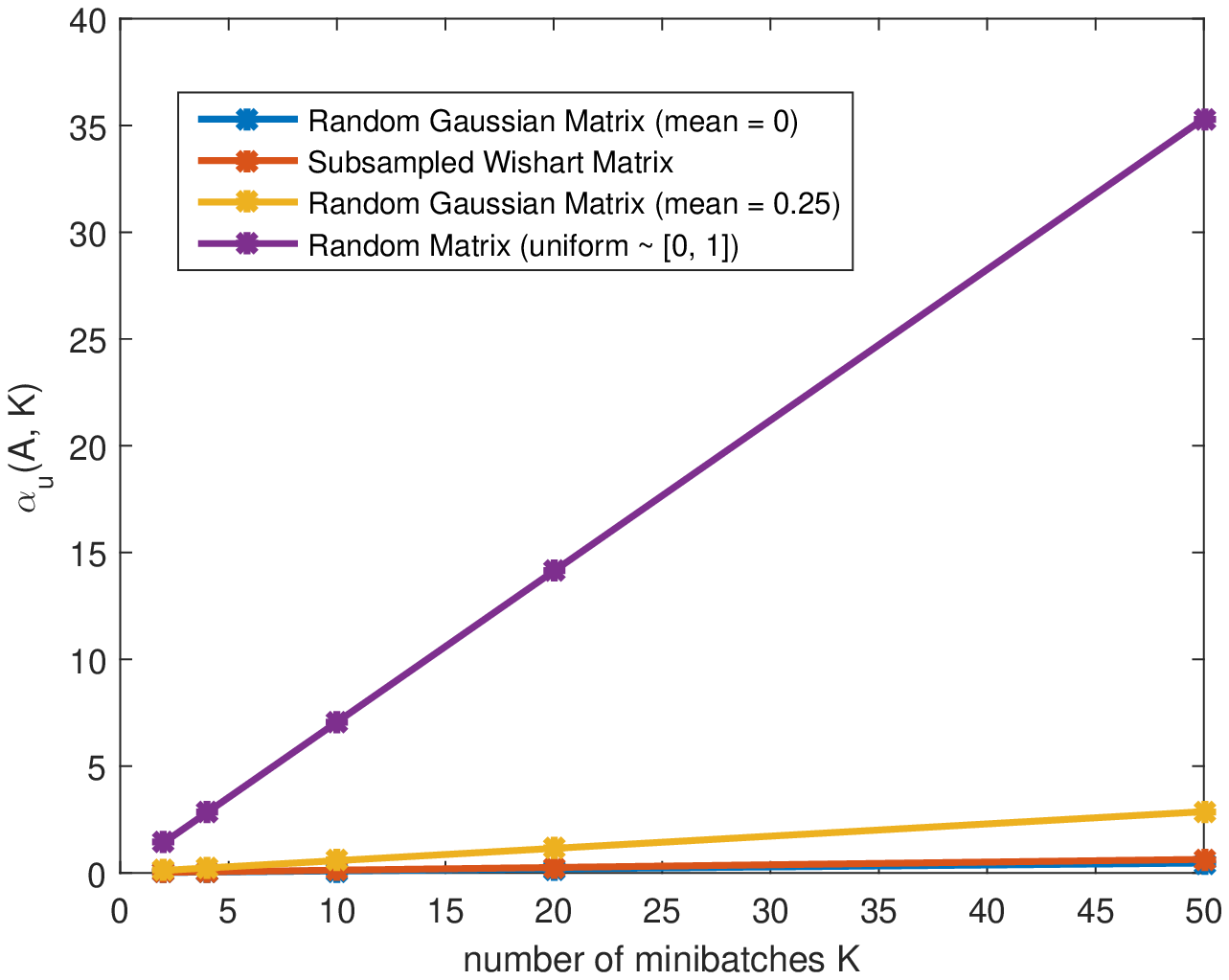}}\\
    {\includegraphics[width=.38\textwidth]{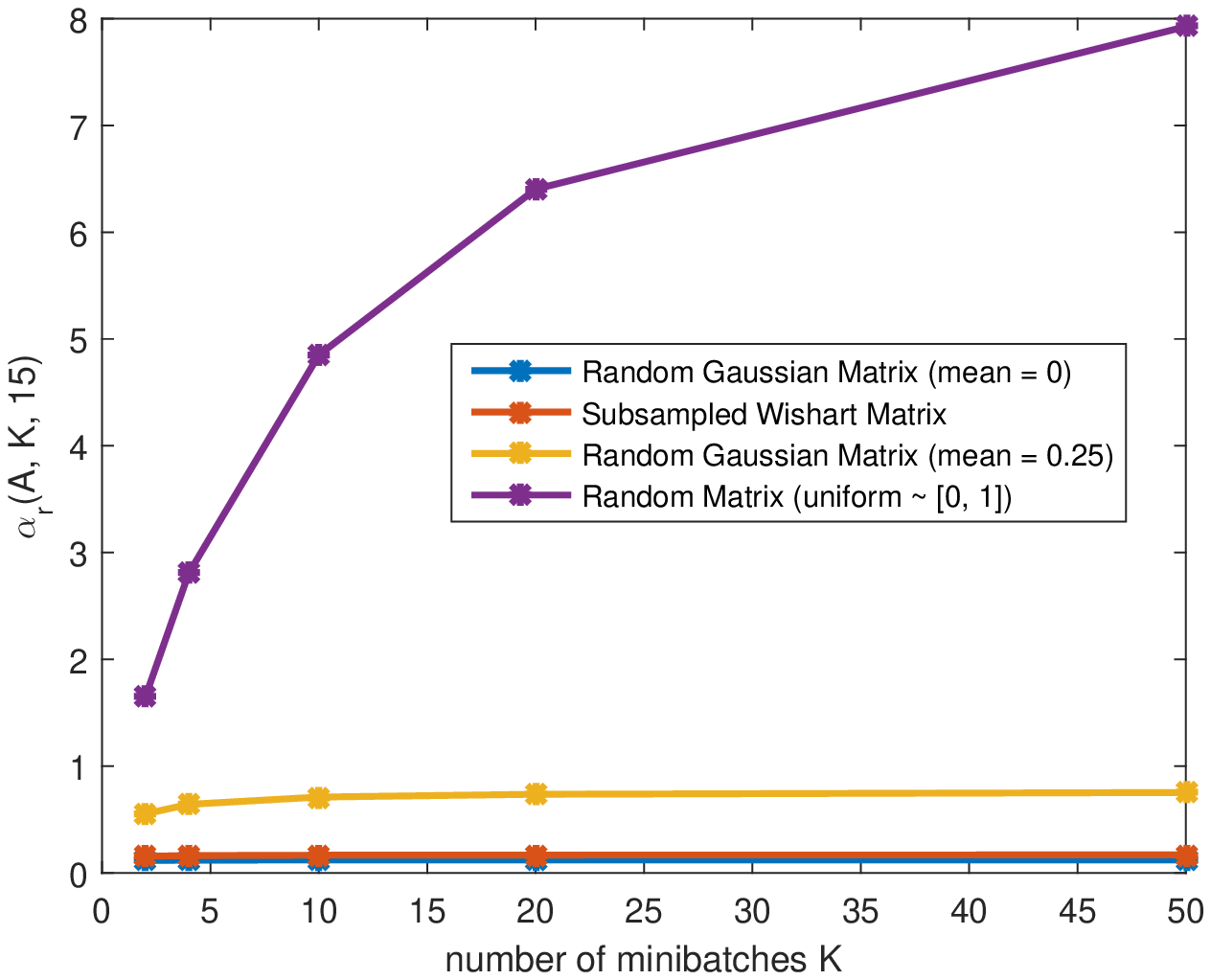}}
    {\includegraphics[width=.38\textwidth]{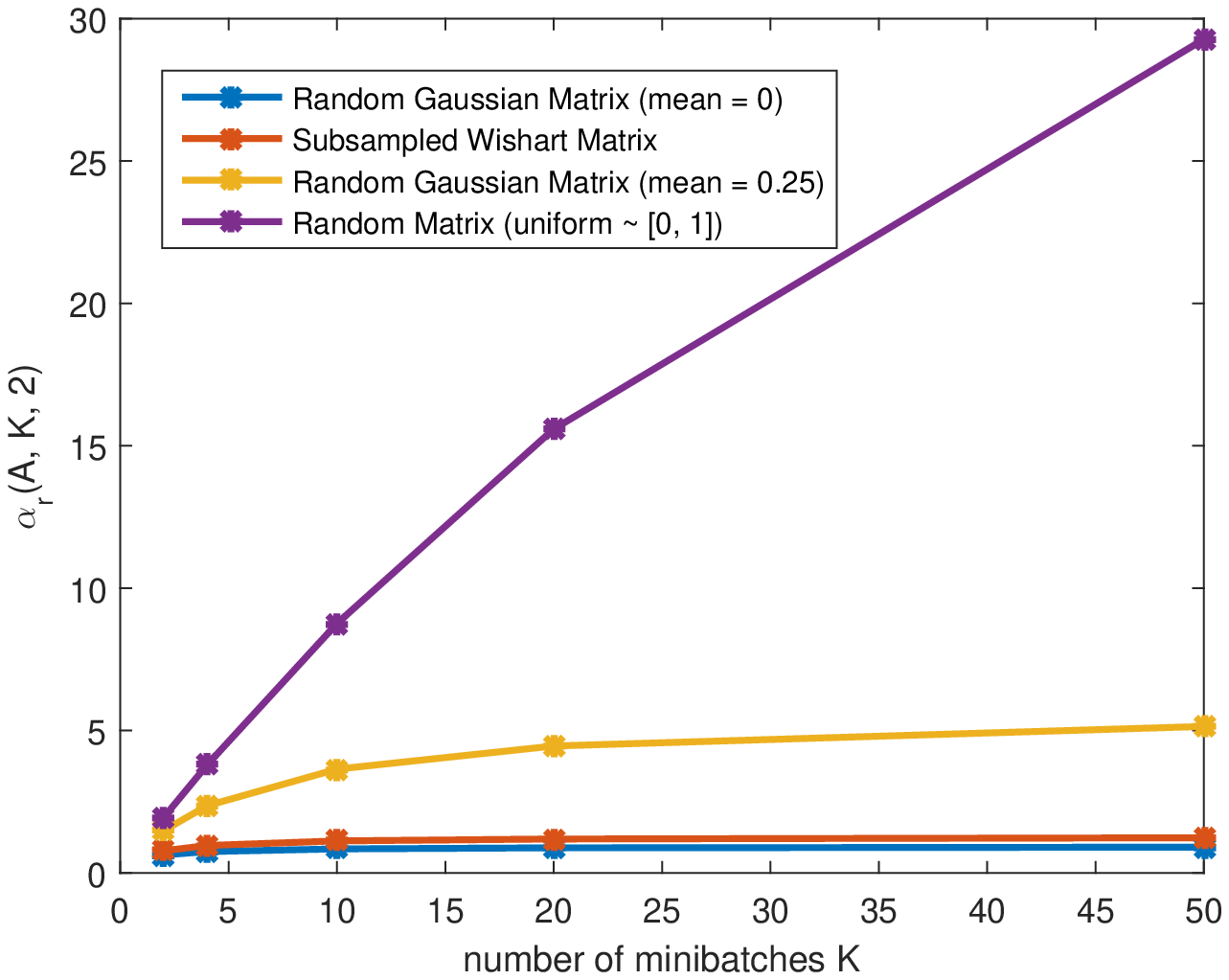}}
	\caption{  Stochastic Acceleration (SA) factors, eigenspectrum, and lower/upper bound estimates for forward operators (random matrices sized 200 by 1000) with different distributions. We use random partitioning here.}
	\label{EA_3f}
\end{figure}

In Figure \ref{EA_2f}, we test three different partition schemes for our running example of space-varying deblurring, the interleaving partitioning which we have described before, the random partitioning where we generate the partition index randomly without replacement, and the consecutive block partitioning\footnote{which is basically $S_k := [m(k-1)+1, m(k-1), m(k-1)+1, ..., mk - 1, mk], \forall k \in [K]$.}, where we directly partition the forward operator $A$ into $K$ consecutive blocks and form the minibatches. We can clearly observe that, the interleaving partitioning in this case provide the smallest local coherence, and hence its SA factors are the largest. While consecutive block partitioning leads to the largest local coherence, and it offers no stochastic acceleration at all. The lower bound estimate $\alpha_\ell(A, \Bar{S}, K)$ of SA factors are actually very accurate in this case, as shown in the Figure \ref{EA_2f}(b).

 \begin{figure}[t] 
	\centering
    {\includegraphics[width=.49\textwidth]{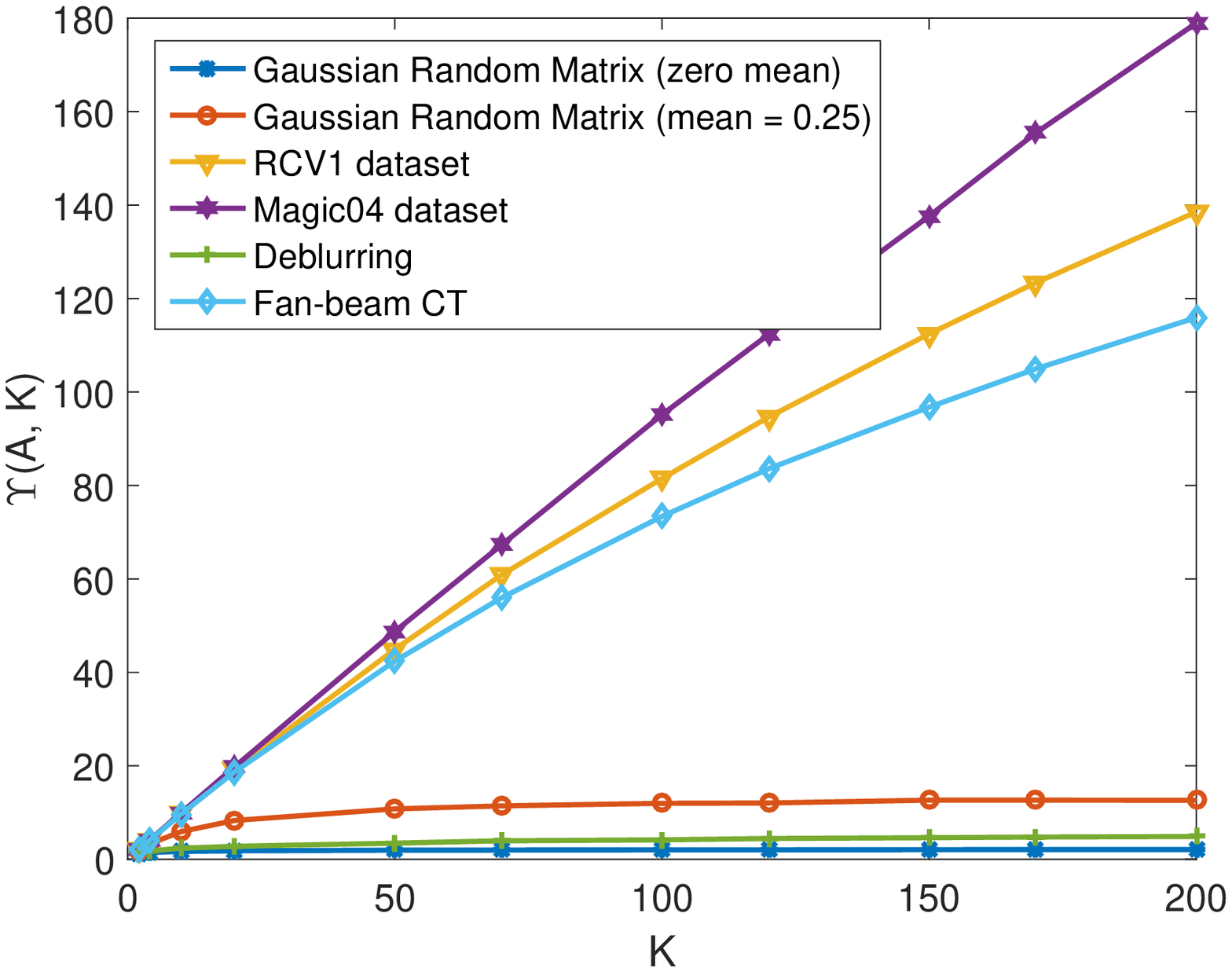}}\\
    {\includegraphics[width=.49\textwidth]{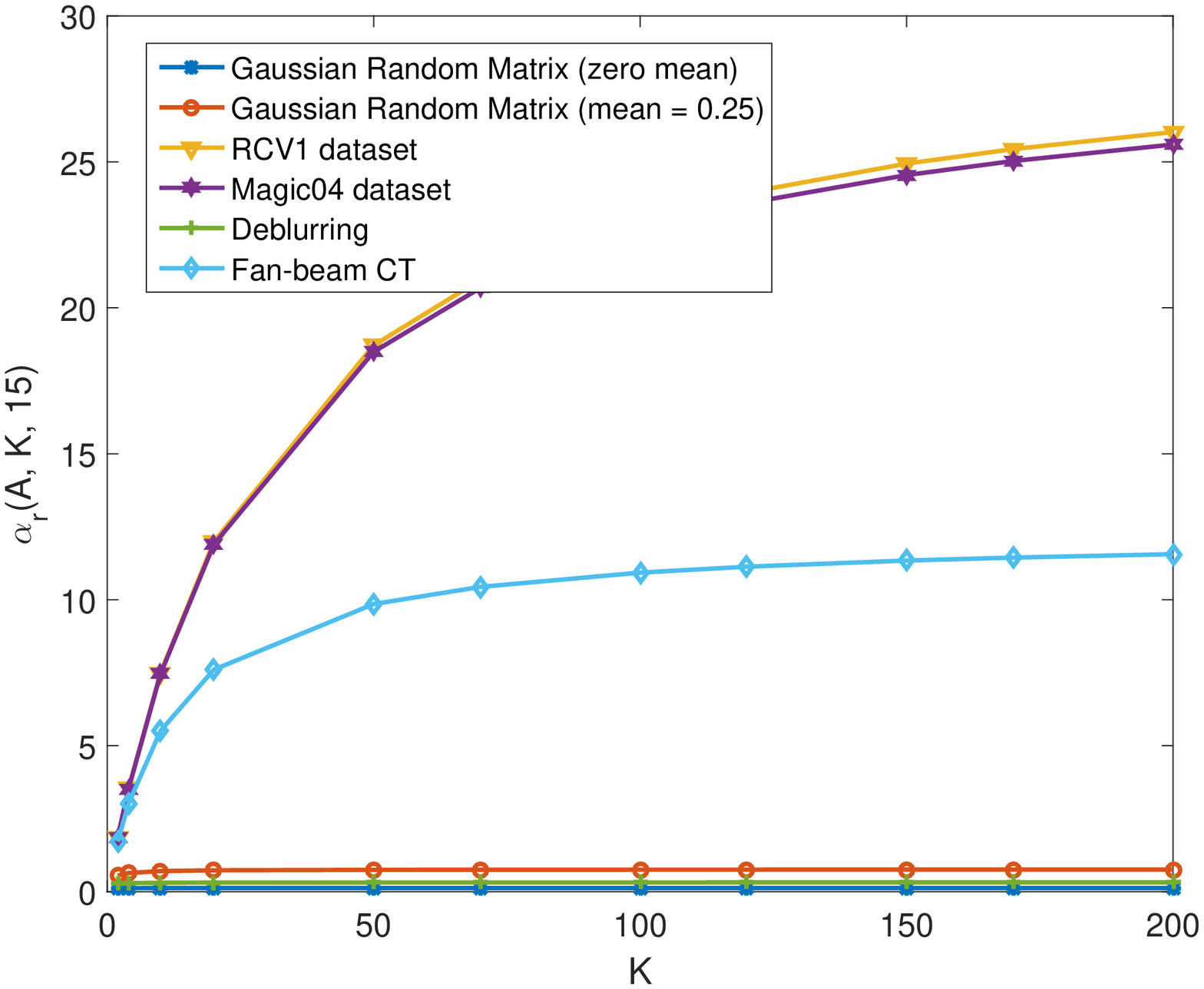}}
    {\includegraphics[width=.49\textwidth]{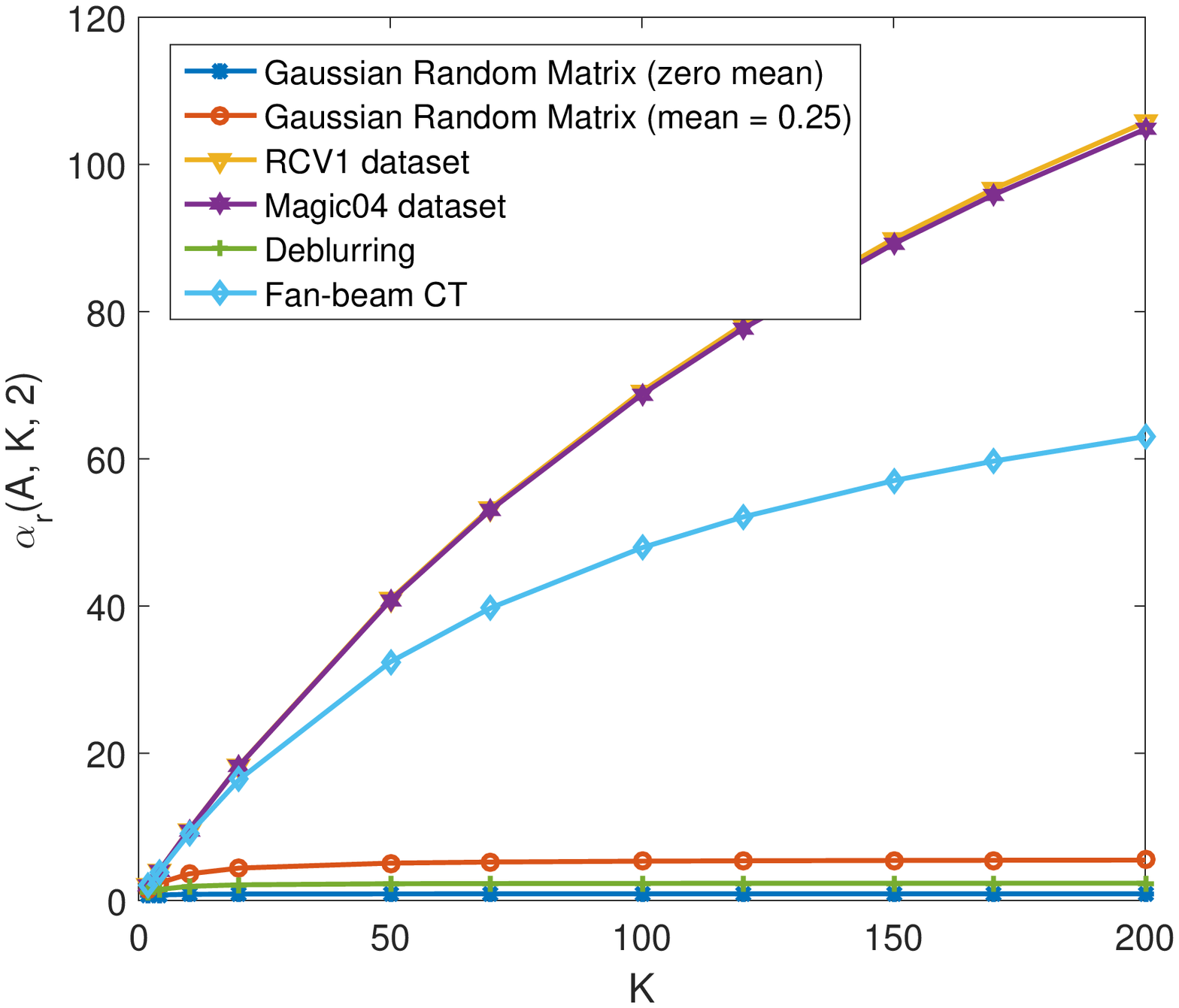}}
	\caption{  SA factors, and lower bound estimates of inverse problems with different forward operators for random partition minibatches.}
	\label{saf_rand_plot}
\end{figure}

In Figure \ref{EA_3f} we present a simulation result where we generate 4 random matrices of the same size $n = 200, d = 1000$ with different distributions and check the relationship of their SA factors and the eigenspectrum of their Hessian matrices. The forward operators we generate are:
\begin{itemize}
    \item (1) random Gaussian matrix with each entry drawn from a Gaussian distribution with zero-mean and unit-variance; 
    \item (2) subsampled Wishart matrix; 
    \item (3) random Gaussian matrix with each entry drawn from a Gaussian distribution with 0.25-mean and unit-variance;
    \item (4) random matrix with each entry drawn from a uniform distribution supported on the interval $[0, 1]$.
\end{itemize}
 From the experimental result we can observe that, these four forward operators have very different decay-rates on their Hessians' eigenspectrum, and correspondingly, very different SA factors. The case (4) has the fastest decay-rate, and has the largest SA factors and it grows almost linearly as the number of minibatches increases. The case (1) has the slowest decay-rate on the eigenspectrum, and correspondingly, it has the worst SA factors among the 4 cases. This numerical result is in broad agreement with our analysis.
 
 We also test our lower bound $\alpha_r(A, K, \delta)$ for all the examples we have considered. We first compute the lower bound estimate $\alpha_r(A, K, \delta)$ by (\ref{lower_bound_rand_par}) for different forward operators, with the choice of $\delta = 15$ which is sufficient for the lower bound to hold with probability at least $0.9$ for all these forward operators. We present the result in Figure \ref{saf_rand_plot}(b) and compare it with the SA factors presented in Figure \ref{saf_rand_plot}(a). We find that our theoretically justified lower bound is still able to distinguish well whether a given inverse problem is suitable or not for stochastic optimization, but seems to be very conservative for the choice of $\delta$. Interestingly, if we heuristically reduce $\delta$ from $15$ to $2$, then we can actually obtain a much better lower bound estimate for the SA factors, as shown in Figure \ref{saf_rand_plot}(c).

\subsubsection{Concluding remark}\label{con_bo} 

In short, the take-home message of our analysis in this subsection contains following aspects:

\begin{itemize}
    \item Firstly, for linear inverse imaging problems, in order to have a good SA factors, the forward operator $A$ should have a small ratio of $\frac{\|A^T\|_{1 \rightarrow 2}^2}{\|A\|^2}$, which essentially means that the Hessian matrix $A^TA$ should have a relatively fast-decaying eigenspectrum. 
    \item Meanwhile, optimizing the choice of partitioning can be crucial in some inverse problems -- a judiciously chosen partition scheme can significantly improve the SA factor if the forward operator has local incoherence structure.
    \item The lower bounds, particularly the one for the random partition\footnote{If one wishes to use the relaxed lower bound estimate, the only need is to compute $\frac{\|A\|^2}{ \|A^T\|_{1 \rightarrow 2}^2}$ which can be very efficiently obtained.} (\ref{lower_bound_rand_par}) can be readily applied by practitioners to conveniently evaluate whether a given inverse problem is suitable for using stochastic gradient methods as solvers. The first lower bound (\ref{lower_bound_SA}) can also be used to compare between different partition schemes and select the best one among them\footnote{One may also achieve this goal by directly computing $L_f$ and $L_b$ with the power method or its accelerated and stochastic variants for each of the compared partition schemes.}. However, finding the best partitioning for a given inverse problem is a combinatorial problem, and we do not currently have any generic scheme which is guaranteed to solve this for arbitrary inverse problems -- we leave this as an open problem for future work. Note that, in this section we have mainly focused on the data-partition minibatch scheme, while we also find that similar results is also valid for random with-replacement minibatch schemes, and we refer the readers to Appendix \ref{Ad4} for details.
\end{itemize}

\section{A practical accelerated stochastic gradient method for imaging}

So far we have considered the role of the Lipschitz constants (and associated step-sizes) in determining the potential advantages of stochastic over deterministic gradient methods in inverse problems. However there are other aspects that complicate the analysis. The most obvious one is that stochastic gradient methods formulated in the primal domain need to calculate the proximal operator many more times than full gradient methods and hence slow down dramatically the run time. There are also scenarios, (see e.g. \cite{ulas2017robust,ulas2017accelerated}), where more than one non-smooth regularization term may be desired, where most of the existing fast stochastic methods such as Katyusha are inapplicable. Here we present a new SGD formalism that aims to mitigate this problem. 

\begin{algorithm}[t]\label{A2}
   \caption{ Accelerated Primal-Dual SGD  (Acc-PD-SGD)}
\begin{algorithmic}[1]
   \State{Initialization: $x^0 = v^0 = v^{-1} \in \mathrm{dom} (g)$, the step size sequences $\alpha_{(.)}, \eta_{(.)}, \theta_{(.)}$, $l = 0$, and the local coherence sampling $\Bar{S}$.}
\For{$t = 1$ {\bfseries to} $N_0$ }
     \State {$x^t \leftarrow \frac{(3t - 2) v^{t - 1} + t x^{t - 1} - (2t - 4) v^{t - 2}}{2t + 2}$, $x_0 \leftarrow x^t$, \\ \ \ \ \ \ $z_0 \leftarrow x^t$, $y_0 \leftarrow Dx_0$  {  \hfill $\rightarrow$ Katyusha-X Momentum}}
     \For{$k = 1$ {\bfseries to} $N_1$ }
      \State{$l \leftarrow  l + 1$}
        \State{$y_{k + 1} = \mathrm{prox}_{\lambda g^*}^{\alpha_l} (y_k + \alpha_l D z_k)$ {  \hfill $\rightarrow$ Dual Ascent}}
        \State{$\mathrm{Pick}$\ \  $i \in [1, 2, ... K]$\ \ $\mathrm{uniformly\ at\ random}$}
        \State{$\triangledown_{k} = \triangledown f_{S_i}(x_k)$} ; 
        \State{$x_{k + 1} = \mathrm{prox}_{\gamma h}^{\eta_l} \left( x_k - \eta_l (D^Ty_{k + 1} + \triangledown_{k} ) \right)$ {  \hfill $\rightarrow$ Primal Descent}}        
        \State{$z_{k + 1} = x_{k + 1} + \theta_l (x_{k + 1} - x_{k})$   {  \hfill $\rightarrow$ Innerloop Momentum}}
     \EndFor
     \State{$ v^t \leftarrow x_K  $}
\EndFor
\State {$\mathrm{Return}\ x^t$}
\end{algorithmic}
\end{algorithm}

We consider in this section the following optimization problem:
\begin{equation}\label{obj_l}
   x^\star  \in  \min_{x \in \mathbb{R}^d} \left\{ \frac{1}{n}\sum_{i = 1}^n f_i(a_i, b_i, x) + \lambda g(Dx) + \gamma h(x)\right\},
\end{equation}
where $f(x) = \frac{1}{n}\sum_{i = 1}^n f_i(a_i, b_i, x)$ is the data-fidelity term, $g(Dx)$ is a regularization term with a linear operator -- for example the TV regularization ($g(.) = \|.\|_1$, $D \in \mathbb{R}^{r \times d}$ is the differential operator), and $h(x)$ is a second convex regularizer. For simplicity of presentation, we only consider here the three-composite case. This is without the loss of generality, since it is well-known in the literature that if an operator-splitting algorithm can solve (\ref{obj_l}) with a certain convergence rate, then it can be immediately extended to tackle multiple non-smooth regularizers by an additional proximal averaging step, with the same convergence rate in the three-composite case \cite{chambolle2016ergodic, zhao2018stochastic, davis2017three}. The saddle-point formulation can be written as:
\begin{equation}\label{saddle}
   [x^\star, y^\star] = \min_{x \in \mathbb{R}^{d}}  \max_{y \in \mathbb{R}^{r}} f(x) + h(x) + y^TDx - \lambda g^*(y)
\end{equation}
The most famous algorithm for solving this saddle-point problem is the primal-dual hybrid gradient (PDHG, also known as the Chambolle-Pock algorithm) \cite{chambolle2011first, chambolle2016ergodic}, which interleaves the update of the primal variable $x$ and the dual variable $y$ throughout the iterates. With this reformulation the linear operator $D$ and the function $g(.)$ are decoupled and hence one can divide-and-conquer the expensive TV-proximal operator with the primal-dual gradient methods. The stochastic variant of the PDHG for the saddle-point problem (\ref{saddle}) has been very recently proposed by Zhao $\&$ Cevher \cite[Alg.1, \say{SPDTCM}]{zhao2018stochastic} and shown to have state-of-the-art performance when compared to PDHG, stochastic ADMM \cite{ouyang2013stochastic} and stochastic proximal averaging \cite{zhong2014accelerated}. We find out that, the SPDTCM method is an efficient and practical optimization algorithm for imaging problems with regularization terms which are coupled with linear operators. Additionally, since the effect of acceleration given by Nesterov's momentum appears to be very important, we also need to consider a way to ensure that our method is accelerated\footnote{We are aware of the recent attempt \cite{zhao2019optimal} to develop accelerated SPDTCM algorithm. However we find this accelerated method is not as practical as SPDTCM and our method in imaging problems since the step-size choices depend on unknown parameters ($\rho$ and $\rho'$ in \cite[Section 2]{zhao2019optimal}) in unconstrained minimization tasks, which need careful tuning manually for different imaging inverse problems in practice. Sub-optimal choices may lead to compromised convergence rates and even diverging behaviors. }. Since the SPDTCM method does not have Nesterov-type acceleration, we propose a variant of it which adopts the outerloop acceleration scheme given by the Katyusha-X algorithm \cite{allen2018katyusha}, which is a simplified variant of the Katyusha algorithm \cite{allen2016katyusha}. We observe that such a momentum step is important for the stochastic primal-dual methods in this application. We present our method as Algorithm 2. One can directly choose the same step-size sequences $\alpha_{(.)}, \eta_{(.)}, \theta_{(.)}$ as suggested in \cite[Section 2.3]{zhao2018stochastic}.

\begin{figure}[t] 
	\centering	
     \includegraphics[width=.77\textwidth]{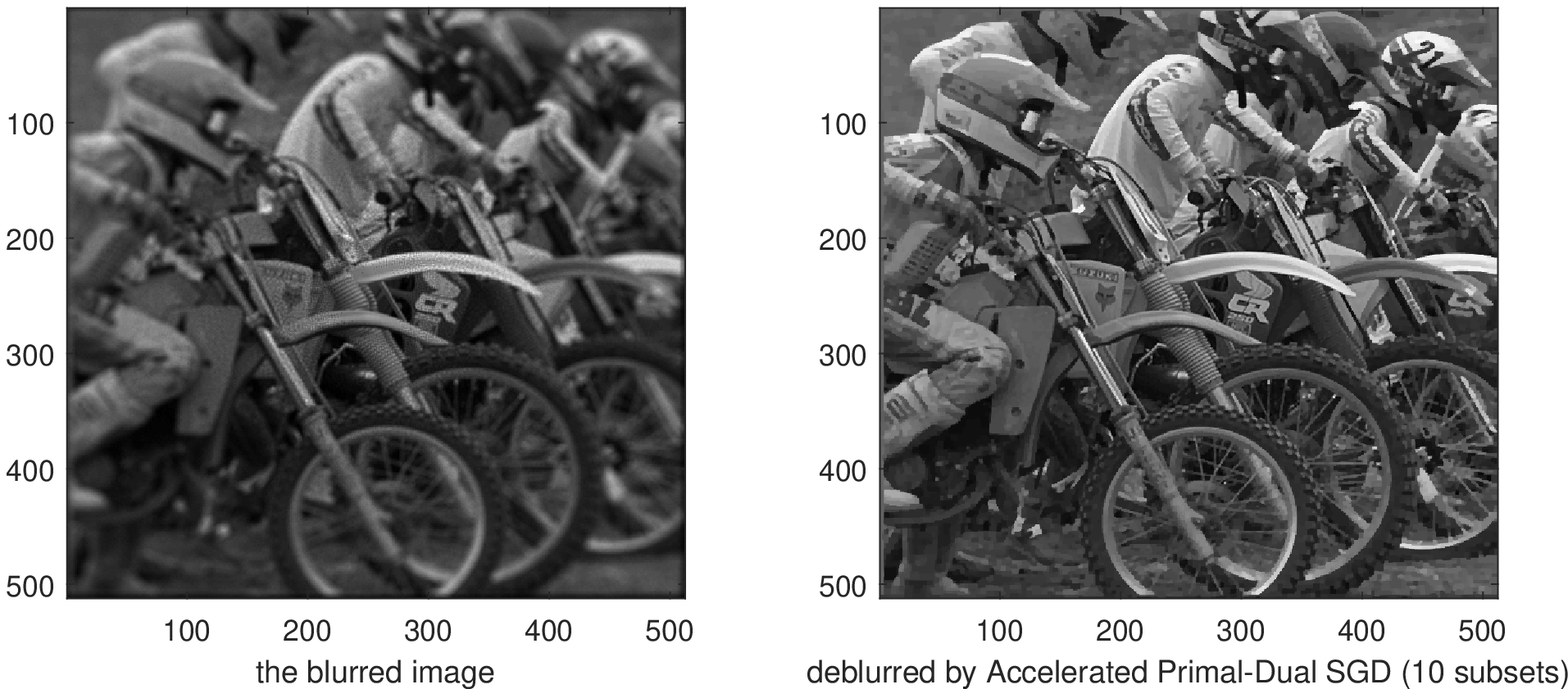}
    \includegraphics[width=.77\textwidth]{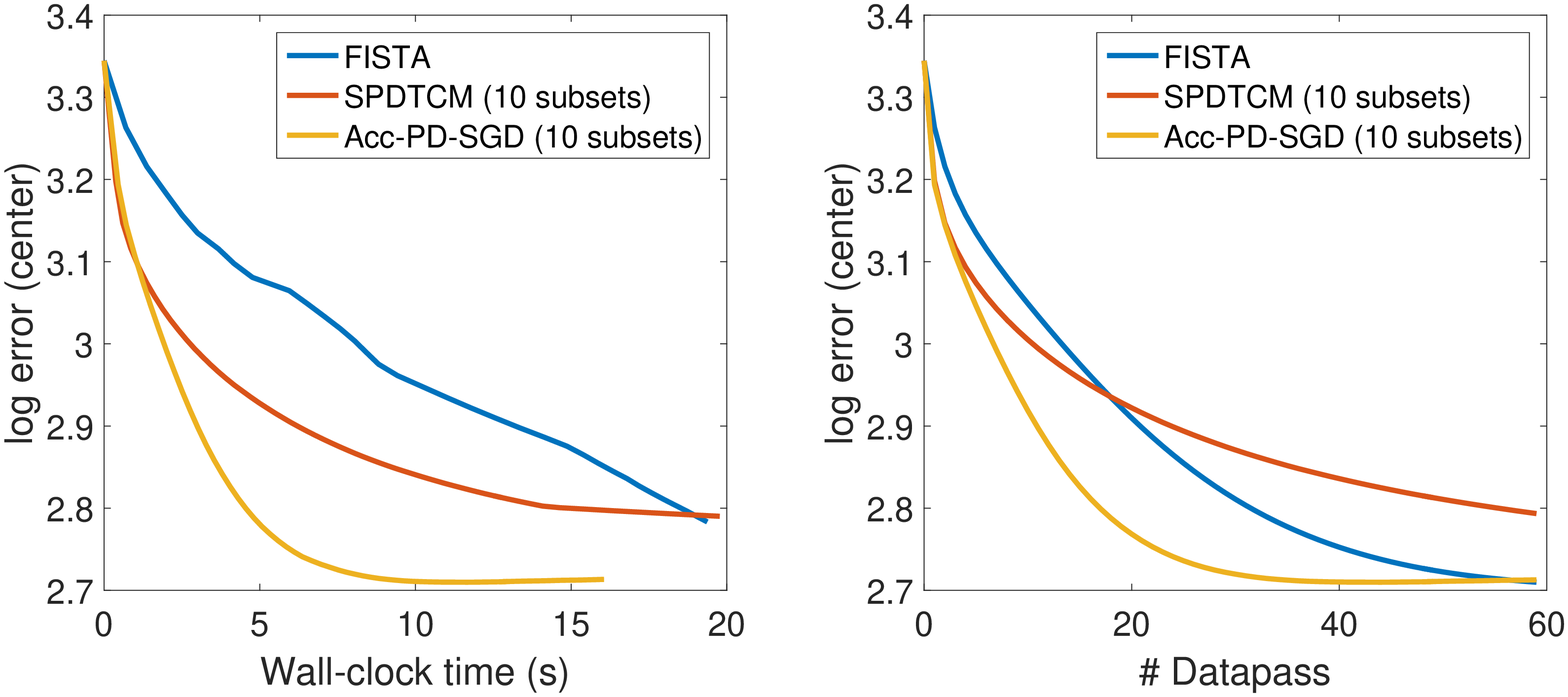}
	\caption{  The estimation error plot for the deblurring experiment with TV-regularization. Image: Kodim05, with an additive Guassian noise (variance 1).}	
	\label{est_error_exp1}
\end{figure}

\section{Numerical experiments}
\subsection{Space-varying image deblurring experiment}
We test our algorithm and compare with FISTA \cite{beck2009fast} and the SPDTCM \cite{zhao2018stochastic} on a space-varying deblurring task for images sized 512 by 512, with a space-varying out-of-focus blur kernel, and TV-regularization. {As suggested by our lower bound $\alpha_\ell(A, \Bar{S}, K)$ in Theorem \ref{Thm1_lower_bound} and the numerical result shown in Figure \ref{EA_2f}, we choose the interleaved subsampling minibatch partition for the deblurring task which has small local-accumulated-coherence for the stochastic algorithms.} All algorithms are initialized with a backprojection. We use a machine with 1.6 GB RAM, 2.60 GHz Intel Core i7-5600U CPU and MATLAB R2015b.

We plot the estimation error $\log_{10}\|x - x^\dagger\|_2^2$ in Figure \ref{est_error_exp1} for each algorithm, where $x^\dagger$ denotes the ground truth image. We observe a roughly $4\times$ improvement in run time compared to FISTA since our algorithm can avoid the heavy cost of the TV proximal operator\footnote{For the computation of the TV proximal-operator for FISTA algorithm, we use the popular implementation from the UnLocBox toolbox \cite{perraudin2014unlocbox}.} while maintaining the fast convergence provided by Nesterov-type momentum and randomization. We also report a significant improvement over the SPDTCM algorithm both in time and iteration complexity. In terms of datapasses (number of epochs), the SPDTCM does not show any advantage over the deterministic method FISTA, while our Acc-PD-SGD with 10 minibatches is able to achieve 2 times acceleration over FISTA. 

We also report that in this experiment, if we further increase the number of subsets of SPDTCM and Acc-PD-SGD, we do not observe faster convergence for these algorithms. In other words, no matter how we increase the number of subsets, this 2-time acceleration (in terms of number of datapasses) is the limit of our algorithm -- such a trend is successfully predicted by the SA factor shown in the Figure \ref{EA_fig}, where we can see that the curve of the SA factor for deblurring task goes flat instead of increasing after the number of minibatches $K > 10$.

\subsection{X-Ray computed tomography image reconstruction experiment}

In the first experiment, we have demonstrated the superior performance of the proposed Acc-PD-SGD algorithm compared to the deterministic algorithm FISTA, and the state-of-the-art stochastic primal-dual gradient method SPDTCM on a space-varying deblurring problem, although it is not inherently favorable for the application of stochastic gradient methods. In this subsection, we turn to another imaging inverse problem -- the computed tomography image reconstruction. As suggested by the curve of the SA factor, the X-ray CT image reconstruction is a nice application for stochastic gradient methods, where we expect them to achieve significant speed-ups over the deterministic methods. 

\begin{figure}[t] 
	\centering	
     \includegraphics[width=.75\textwidth]{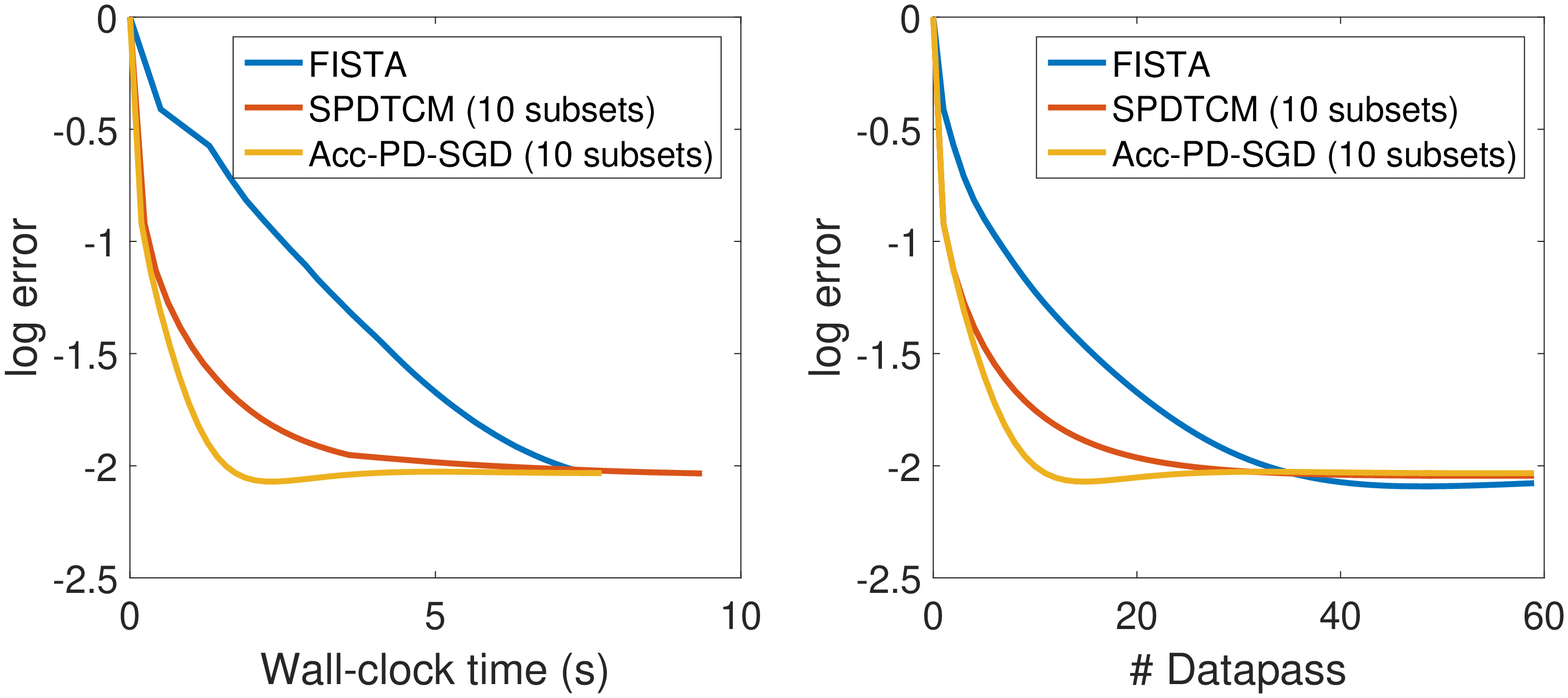}

	\caption{  The estimation error plot for the X-ray CT image reconstruction experiment with TV-regularization. $\log_{10}\frac{\|Ax^\dagger\|_2^2}{\|w\|_2^2} \approx 3.16$.}		
	\label{ct_mod}
\end{figure}

\begin{figure}[h] 
	\centering	

    \includegraphics[width=.75\textwidth]{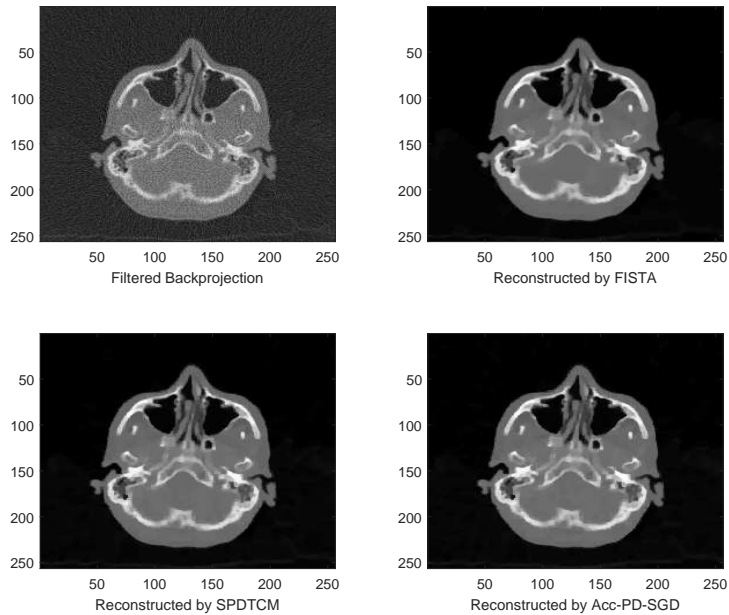}
	\caption{  The reconstructed images by the compared algorithms with TV-regularization.}	
	\label{ct_ima_mod}	
\end{figure}

\begin{figure}[t]
	\centering	
     \includegraphics[width=.75\textwidth]{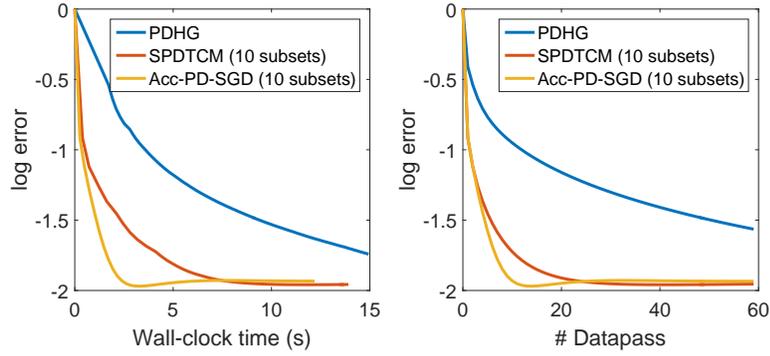}

	\caption{  The estimation error plot for the X-ray CT image reconstruction experiment with TV-regularization and $\ell_1$ regularization on Haar-wavelet basis. $\log_{10}\frac{\|Ax^\dagger\|_2^2}{\|w\|_2^2} \approx 2.86$.}	
	\label{ct_mod2} 	
\end{figure}

\begin{figure}[h] 
	\centering	

    \includegraphics[width=.75\textwidth]{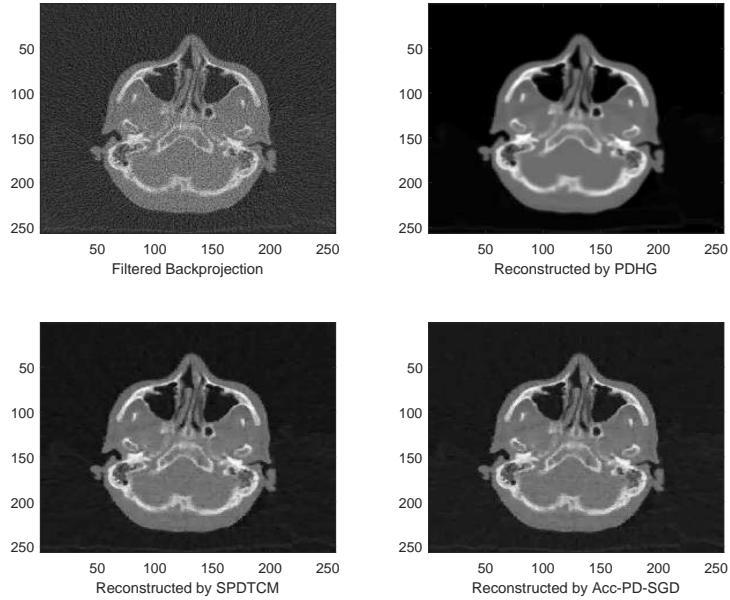}
	\caption{  The reconstructed images by the compared algorithms at termination using joint TV-$\ell_1$ regularization.}		\label{ct_mod2_ima}
\end{figure}

In this experiment we consider a 2D fan-beam CT imaging problem generated via the Matlab package {\it AIRtools} \cite{hansen2012air}, where we aim to reconstruct a $256 \times 256$ head image from $92532$ noisy X-ray measurements (hence the forward operator $A \in \mb{R}^{92532 \times 65536}$), using TV-regularization. Denoting $x^\dagger$ to be the (vectorized) ground truth image and $w \in \mb{R}^n$ to be an additional random noise vector drawn from an exponential Poisson distribution, we have the observed measurement as $b = Ax^\dagger + w$. The signal-to-noise ratio of the X-ray measurement in this example is set to be: $\log_{10}\frac{\|Ax^\dagger\|_2^2}{\|w\|_2^2} \approx 3.16$. We present the convergence results of the compared algorithms in Figure \ref{ct_mod}. In Figure \ref{ct_ima_mod}, we first demonstrate the reconstructed image by the classic filtered-backprojection (FBP) \cite{basu2000n} which is a direct method without considering regularization. From the result of FBP, we can clearly see that the reconstructed image contains a large amount of noise.

From the convergence results of the iterative algorithms in Figure \ref{ct_mod} we can observe that for this experiment, the stochastic methods SPDTCM and Acc-PD-SGD converges significantly faster than the full gradient method FISTA both in terms of number of datapasses and wall-clock time. Meanwhile, we can also see that, our proposed method Acc-PD-SGD converges faster than the SPDTCM which does not use Katyusha-X momentum for acceleration.

Moreover, in some scenarios (such as the cases where we use low-dose X-ray measurements), we may wish to use more than just one regularizer for a better modelling of the ground truth, in order to ensure an accuracy estimation via additional prior information. In the following experiment, we reduce a half of the dosage of the X-ray measurement, such that $\log_{10}\frac{\|Ax^\dagger\|_2^2}{\|w\|_2^2} \approx 2.86$, and use two regularization terms jointly for the reconstruction task -- the TV regularization and $\ell_1$ regularization on the Haar-wavelet basis. The FISTA algorithm cannot be directly applied for this three-composite optimization task, hence we choose the Chambolle-Pock (PDHG) algorithm as a baseline representing the full gradient methods, and SPDTCM as the representative baseline for the state-of-the-art stochastic gradient methods for the three-composite problems. We present the results of this experiment in Figure \ref{ct_mod2} and \ref{ct_mod2_ima}, where we can again clearly observe the superior performance of the proposed method over the baselines on this experiment where we use multiple non-smooth regularization terms.

\section{Conclusion}
In this work we began by investigating the practicability of the state-of-the-art stochastic gradient methods in imaging inverse problems. We first presented a surprisingly negative result on existing SGD-type methods on image deblurring, as a motivational example. To understand the limitation of stochastic gradient methods in inverse problems, we have provided a novel analysis for the estimation-error convergence rate of minibatch SGD in the setting of linear inverse problem with constraints as regularization, under restricted strong-convexity \cite{agarwal2012fast, 2015_Oymak_Sharp} and expected smoothness \cite{gower2018stochastic} conditions. Based on our theoretical analysis, we have proposed the SA factor to evaluate the possible computational advantage of using stochastic techniques for a given task. Then we went further and made an in-depth analysis for understanding how the inherent structure of the forward measurement model can contribute to the practicality of the stochastic gradient methods with partition minibatch schemes. 

We also derived lower and upper bounds of the SA factor. From the theoretical results, we find out that, if an linear inverse problem has a small ratio of $\frac{\|A^T\|_{1 \rightarrow 2}^2}{\|A\|^2}$, which means the Hessian matrix $A^TA$ has fast-decaying eigenspectrum, then it typically admit good SA factors, hence can be rapidly solved by stochastic gradient methods. Our analysis also suggests that, excellent partition schemes typically have low local-accumulated-coherence, which essentially means the measurements within one minibatch are mutually less correlated. Using our SA factor, jointly with the derived lower bounds, the practitioners can easily identify whether they should use stochastic gradient or deterministic gradient algorithms for given inverse problems, and evaluate the potential of given partition schemes.

While our results are mainly for linear inverse problems with least-squares data-fidelity terms and convex regularizers, we believe that they also can be extended and give insights to non-linear inverse problems since one can construct majorizing linearized subproblems (proximal Newton-steps) and solve these subproblems with deterministic or stochastic proximal gradient methods. Our results can also be extended for understanding and analyzing the limitations of stochastic gradient-based methods \cite{sun2019online, sun2019block} with the plug-and-play priors \cite{venkatakrishnan2013plug, kamilov2017plug} and regularization-by-denoising schemes \cite{romano2017little, reehorst2018regularization} in imaging inverse problems, which we leave as a future direction.

Although we have concentrated on stochastic gradient methods vs deterministic gradient methods, there are other considerations that might affect the choice of whether to go stochastic. For example, if an inverse problem can be effectively preconditioned by simple preconditioners (such as diagonal preconditioners), then the potential benefit of stochastic methods over deterministic methods may possibly be reduced, since the preconditioned forward operator may not have as fast-decaying spectrum as the original one. Moreover, if the forward operator can be implemented with a fast transform such as the FFT, for example in MRI image reconstruction tasks, the deterministic gradient methods are usually much more favored since they can benefit from the fast operation while current stochastic gradient methods cannot.

Finally, as a side contribution, we propose the Accelerated Primal-Dual SGD to cope with multiple regularizers (potentially) with a linear operator while maintaining the fast convergence, and demonstrate its effectiveness via experiments on space-varying deblurring and X-Ray CT image reconstruction. Although we have not yet done the theoretical convergence analysis of the Acc-PD-SGD algorithm, we believe it provides insights for the algorithmic design of fast stochastic gradient methods tailored specifically for imaging inverse problems, from understanding the inherent limitation, to the practical algorithmic framework.

\appendix

\section{Lower bounds for deterministic gradient and stochastic gradient optimization}\label{Ad1}
We present some well-known lower bounds for first-order optimization. We start by the lower-bound derived by Nesterov \cite[Theorem 2.16]{nesterov2013introductory}:
\begin{theorem}\label{lower_bound_full_gradient}
(Lower-bound for convex and smooth optimization \cite{nesterov2013introductory}) For any $1 \leq k \leq \frac{1}{2}( d - 1 )$ and $x^0 \in \mb{R}^d$ there exist a function $F \in \mc{F}_L^{1, 1}(\mb{R}^d)$ such that for any iterative algorithm which uses only first-order oracle $\triangledown F(.)$, the following inequality holds:
\begin{equation}
    F(x^k) - F^\star \geq \Omega \left(\frac{L_f\|x^0 - x^\star\|_2^2}{ (k + 1)^2} \right).
\end{equation}
\end{theorem}
Such a lower-bound suggests that there exists at least one $L$-smooth convex function on which any first-order method cannot converge faster than $O(1/k^2)$ for a limited number of iterations which $1 \leq k \leq \frac{1}{2}( d - 1 )$.

For the stochastic gradient-based optimization, several researchers \cite{woodworth2016tight, lan2015optimal} have derived important lower-bounds for optimizing the finite-sum objective with stochastic gradient oracle $\triangledown f_i(.)$, and we present here a typical well-known result:
 \begin{theorem}\label{lower_bound_SGD}
 (Lower bound for convex and smooth finite-sum optimization \cite[Theorem 7]{woodworth2016tight}.) For any randomized algorithms with access to the stochastic gradient oracle $\triangledown f_i(.)$, and any $L$, $R$, $\epsilon \geq 0$, there exist a sufficiently large dimension $d = O(\frac{L^2R^6}{\epsilon^2} \log\frac{LR^2}{\epsilon} + R^2 n \log n)$, and $n$ functions $f_i \in \mc{F}_L^{1,1}(\mc{X})$ where $\mc{X} \in \left\{x \in \mb{R}^d | \|x\|_2 \leq R \right\}$, such that in order to achieve an output $\Hat{x}$ which satisfies $\mb{E}[F(\Hat{x}) - F(x^\star)] \leq \epsilon$ for the minimization task:
 \begin{equation}
     x^\star \in \arg \min_{x \in \mc{X}} \bigg\{F(x) := \frac{1}{n}\sum_{i = 1}^n f_i(x) \bigg\},
 \end{equation}
 a necessary 
 \begin{equation}\label{lower_bound_stochastic_gradient}
     \Omega \left(n + R\sqrt{\frac{nL}{\epsilon}}\right)
 \end{equation}
 number of stochastic gradient evaluations are needed.
 \end{theorem}

\section{The Proof of Theorem \ref{thm_6.3.1}}\label{Ad2}

The proof of this theorem is straight forward and is based on combining the existing results since we have assumed the dimension $d$ is large enough for the lower-bound for stochastic gradient oracles to hold on a domain:
\begin{equation}
    \X = \left\{ x \in \mathbb{R}^d : \|x\|_2^2 \leq 1 \right\}.
\end{equation}
According to the lower bound for the stochastic gradient we have presented in Theorem \ref{lower_bound_SGD} \cite[Theorem 7]{woodworth2016tight}, there exists a positive constant $C_{\mathrm{stoc}}$, which is independent of $L_b$, $l_f$ and $K$, such that in order to achieve an output $ \E f(x_\A^s) - f^\star \leq \epsilon$, any stochastic gradient algorithm must take at least:
\begin{equation}
     C_{\mathrm{stoc}}\left(K + \sqrt{\frac{KL_b}{\epsilon}}\right)
\end{equation}
calls of the stochastic gradient oracle $\triangledown f_i()$. In other words, for this worst case function, if we run any stochastic gradient method with only $Ks$ calls on the stochastic gradient oracle such that:
\begin{equation}
    Ks =  C_{\mathrm{stoc}}\sqrt{\frac{KL_b}{\epsilon}},
\end{equation}
$\E f(x^s_{\A_\mr{stoc}}) - f^\star \geq \epsilon$ can be guaranteed. Hence, we have:
\begin{equation}
    \E f(x^s_{\A_\mr{stoc}}) - f^\star \geq \frac{C_{\mathrm{stoc}}^2 L_b}{Ks^2}
\end{equation}

Meanwhile, starting from $x^0 \in \mathcal{X}$, by Def. \ref{class_of_optimal_full}, for any optimal full gradient method $\A_\mr{full}$ we can have:
\begin{equation}
  f(x^s_{\A_\mr{full}}) - f^\star \leq \frac{C_1 L_f\|x^0 -x^\star\|_2^2}{s^2} \leq \frac{4C_1 L_f}{s^2},
\end{equation}
where the constant $C_1$ is independent of $L_b$, $L_f$ and $K$. Combining these two bounds we can have:
\begin{equation}
     \frac{\E f(x^s_{\A_\mr{stoc}}) - f^\star}{f(x^s_{\A_\mr{full}}) - f^\star} \geq \frac{C_{\mathrm{stoc}}^2 L_b}{4C_1 KL_f}.
\end{equation}
Finally, by setting $c_0 = \frac{C_{\mathrm{stoc}}^2}{4C_1}$ we yield the claim.

\section{The proof of Theorem \ref{T1}: convergence of minibatch SGD on constrained Least-squares}\label{Ad3}

We first present the following lemma:
\begin{lemma}\label{expect_sm}
For $\|w\|_2 = 0$, we have:
\begin{equation}
    \mathbb{E}_S(\|\frac{1}{m}A^TS^TSA(x - x^\dagger)\|_2^2) \leq \frac{L_e}{n} \|A(x - x^\dagger)\|_2^2,
\end{equation}
\end{lemma}
\begin{proof}
Due to the definition of expected smoothness (\ref{es_def}), we have:
\begin{equation}
    \mathbb{E}_S(\|\frac{1}{m}A^TS^TSA(x - y)\|_2^2) \leq 2 L_e (\frac{1}{2n}\|Ax - b\|_2^2 - \frac{1}{2n}\|Ay- b\|_2^2 - \langle \triangledown f(y), x-y \rangle).
\end{equation}
Now set $y = x^\dagger$, and since $\langle \triangledown f(x^\dagger), x-x^\dagger \rangle = 0$ at the noiseless case, we have:
\begin{equation}
    \mathbb{E}_S(\|\frac{1}{m}A^TS^TSA(x - x^\dagger)\|_2^2) \leq 2 L_e (\frac{1}{2n}\|Ax - b\|_2^2 - \frac{1}{2n}\|Ax^\dagger - b\|_2^2).
\end{equation}
Note that by definition $b = Ax^\dagger + w$, and since it is assumed here $\|w\|_2 = 0$, we have:
\begin{equation}
    \mathbb{E}_S(\|\frac{1}{m}A^TS^TSA(x - x^\dagger)\|_2^2) \leq \frac{L_e}{n} \|A(x - x^\dagger)\|_2^2 .
\end{equation}
Thus finishes the proof of this Lemma.
\end{proof}
Now we present the complete proof of Theorem \ref{T1}:
\begin{proof}
For iteration index $i$ of minibatch SGD we have the following:
\begin{eqnarray*}
\|x^{i+1}-x^\dagger\|_2
&\leq& \|\mathcal{P_{\mathcal{K}}}(x^i-\eta\cdot\frac{1}{m} A^T{S^i}^T{S^i}(Ax^i-b))-x^\dagger\|_2\\
(a)&=& \|\mathcal{P}_{\mathcal{K}-x^\dagger}(x^i-x^\dagger-\eta\cdot\frac{1}{m} A^T{S^i}^T{S^i}(Ax^i-b))\|_2\\
&=& \|\mathcal{P}_{\mathcal{K}-x^\dagger}(x^i-x^\dagger - \eta\cdot\frac{1}{m} A^T{S^i}^T{S^i}(Ax^i-Ax^\dagger))\|_2\\
(b)&\leq& \|\mathcal{P}_{\mathcal{C}}(x^i-x^\dagger - \eta\cdot\frac{1}{m} A^T{S^i}^T{S^i}(Ax^i-Ax^\dagger))\|_2\\
(c)&=& \sup_{v \in \mathcal{C} \cap \mathcal{B}^{n}} v^T(x^i-x^\dagger - \eta\cdot\frac{1}{m} A^T{S^i}^T{S^i}(Ax^i-Ax^\dagger))\\
&\leq& \|x^i-x^\dagger - \eta A^TS^TSA(x^i-x^\dagger)\|_2+(\sup_{v \in \mathcal{C} \cap \mathcal{B}^{n}} v^TA^TS^TS\frac{w}{\|w\|_2})\|w\|_2\\
(d)&\leq& \|(I - \eta \cdot \frac{1}{m}A^T{S^i}^T{S^i}A)(x^i-x^\dagger)\|_2
\end{eqnarray*}
Line (a) holds because of the distance preservation of translation \cite[Lemma 6.3]{2015_Oymak_Sharp}; line (b) holds because of the length of the projection onto a convex set which includes $\bold{0}$ is smaller than the length of projection onto a cone containing the set \cite[Lemma 6.4]{2015_Oymak_Sharp}; line (c) holds because of the definition of the cone-projection operator \cite[Lemma 6.2]{2015_Oymak_Sharp}. Line (d) holds because of the non-expansiveness of the cone-projection operator. Now we take the expectation of $\|x^{i+1}-x^\dagger\|_2$ over the randomness of minibatch sampling at iteration $i$, and consequently:
\begin{eqnarray*}
&&\E(\|x^{i+1}-x^\dagger\|_2)\\
&\leq& \E(\|(I - \eta \cdot \frac{1}{m}A^T{S^i}^T{S^i}A)(x^i-x^\dagger)\|_2)\\
(e)&\leq& \sqrt{ \E(\|(I - \eta \cdot \frac{1}{m}A^T{S^i}^T{S^i}A)(x^i-x^\dagger)\|_2^2)}\\
&=& \sqrt{ \E( \|x^i-x^\dagger\|_2^2-2\eta\cdot\frac{1}{m}\|{S^i}A(x^i-x^\dagger)\|_2^2+\eta^2\| \frac{1}{m}A^T{S^i}^T{S^i}A(x^i-x^\dagger)\|_2^2)}\\
(f)&\leq& \sqrt{  \|x^i-x^\dagger\|_2^2-2\eta\cdot\frac{1}{n}\|A(x^i-x^\dagger)\|_2^2+\eta^2 (\frac{L_e}{n} \|A(x^i - x^\dagger)\|_2^2 )}\\
&\leq& \sqrt{  \|x^i-x^\dagger\|_2^2-(2\eta-L_e\eta^2)\cdot \frac{1}{n} \|A(x^i - x^\dagger)\|_2^2}\\
(g)&\leq& \sqrt{  \|x^i-x^\dagger\|_2^2-(2\eta\mu_c - L_e \mu_c\eta^2 ) \|x^i - x^\dagger\|_2^2}\\
&=&\sqrt{1- 2\mu_c\eta+ L_e \mu_c\eta^2}\|x^i-x^\dagger\|_2
\end{eqnarray*}

Line (e) uses the Jensen's inequality, line (f) is due to Lemma \ref{expect_sm}, while inequality (g) holds because of the restricted strong-convexity condition (Def. \ref{D5}), and we choose $\eta \leq \frac{2}{L_e}$.  Then because the subsampling in each iteration is independent from the previous error vector, by the tower rule we yield:
 \begin{eqnarray*}
     \E(\|x^i-x^\dagger\|_2) &\leq& (1-2\mu_c\eta+ L_e \mu_c\eta^2)^{\frac{i}{2}}\|x^0-x^\dagger\|_2.
 \end{eqnarray*}
Thus we finish the proof by choosing $\eta = \frac{1}{L_e}$.
\end{proof}

\section{Estimating the stochastic acceleration for random with-replacement sampling schemes}\label{Ad4}

The notion of stochastic acceleration factor can also be extended to the case where we use random with-replacement sampling scheme. For random with-replacement minibatch scheme, \cite[Proposition 3.8]{pmlr-v97-qian19b} shows that the expected smoothness constant $L_e$ in (\ref{es_def}) can be upper bounded by:
\begin{equation}
    L_e \leq \frac{m - 1}{m(n - 1)}L_f + \frac{n-m}{m(n-1)}\max_{i \in [n]} \|a_i\|_2^2
\end{equation}

 Recall our Remark \ref{remark_1} comparing deterministic gradient descent and minibatch SGD on constrained least-squares. Denote here $L_f = \frac{1}{n}\|A^TA\|_2$. If $\|w\| = 0$ and, to guarantee an estimation accuracy $\|x_K - x^\dagger \|_2 \leq \varepsilon$, the deterministic proximal gradient descent needs:
 \begin{equation}
     N_{\mr{full}} = \frac{L_f}{\mu_c} \log\frac{\|x_0 - x^\dagger\|_2}{\varepsilon} 
 \end{equation}
iterations, while the minibatch SGD needs:
 \begin{equation}
     N_{\mr{stoc}} = \frac{2L_e}{\mu_c} \log\frac{\|x_0 - x^\dagger\|_2}{\varepsilon},
 \end{equation}
to achieve $\E\|x_K - x^\dagger \|_2 \leq \varepsilon$. Hence the iteration complexity of minibatch SGD is $\Upsilon_e(A, m)$-times smaller than the proximal gradient descent, where:
\begin{equation}
    \Upsilon_e(A, m) = \frac{\frac{n}{m} N_{\mr{full}}}{N_{\mr{stoc}}} = \frac{\frac{n}{2m}L_f}{L_e} \geq \frac{1}{\frac{m - 1}{2(n - 1)} + \frac{n - m}{2(n-1)}\frac{\|A^T\|_{1 \rightarrow 2}^2}{\|A\|^2}}.
\end{equation}
We name $\Upsilon_e(A, m)$ as the expected SA factor for the random with-replacement sampling scheme.

\begin{figure}[t] 
	\centering	
    \includegraphics[width=.57\textwidth]{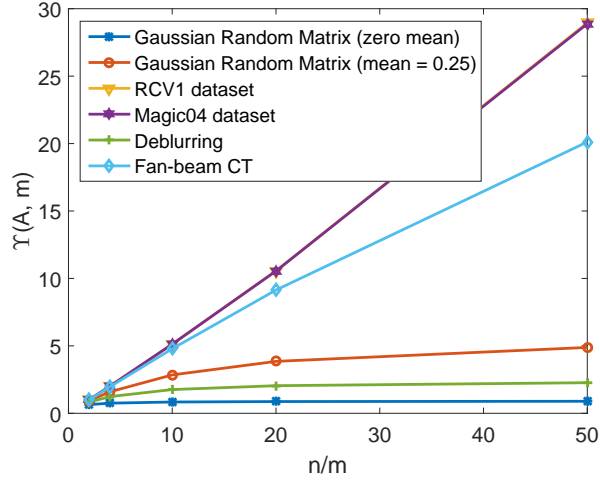}
	\caption{  Expected SA factor of inverse problems with different forward operators for uniform random sampling with replacement.}
	\label{EA_fig_rand_wp}
\end{figure}

\begin{remark}
From the expression of $\Upsilon_e(A, m)$ we can find that the key factor which influences the advantage of stochastic gradient over determinstic gradient is again the ratio $\frac{\|A^T\|_{1 \rightarrow 2}^2}{\|A\|^2}$ which also occurs in our lower bounds for data-partition minibatch schemes in Theorem \ref{Thm1_lower_bound} and \ref{Thm_rand par}.

Our results in this section regarding random with-replacement sampling have several restrictions. So far we can establish results only for minibatch SGD algorithm and meanwhile $\Upsilon_e(A, m)$ is derived based on comparing this result with the iteration complexity of proximal gradient descent. If we use $\Upsilon_e(A, m)$ to measure whether an inverse problem is suitable for stochastic gradient algorithms using a random with-replacement minibatch scheme, we are implicitly making the conjecture that all these minibatch proximal stochastic gradient methods including variance-reduced methods such as SVRG, SAGA and Katyusha, admit the large step-size choices $O(1/L_e)$ based on the expected smoothness constant $L_e$, which has not been shown yet in the literature.

In Figure \ref{EA_fig_rand_wp}, we plot the expected SA factor $\Upsilon_e(A, m)$ for a range of inverse problems which we have considered so far. Similar to the results shown in Figure \ref{EA_fig}, we observe that, the space-varying deblurring example and the zero-mean random Guassian example have the worst expected SA factors. 
\end{remark}

\section{The proof of Theorem \ref{Thm1_lower_bound}}\label{Ad5}

Now we present the proof of Theorem \ref{Thm1_lower_bound}.
\begin{proof}
If we set $H = S^kA (S^k A)^T$ for some $k \in [K]$, the top eigenvalue of $S^kA (S^k A)^T$ is no larger than the largest value within the set $G(S^kA (S^k A)^T)$ which we denote here as $G_{\max}(S^kA (S^k A)^T)$. We have the following relationship:
\begin{equation}\label{GDT_co}
    \|S^kA\|^2 \leq  G_{\max} (S^kA (S^k A)^T) = \max_{i \in S_k} \|(S^kA) a_i\|_1 = \max_{i \in S_k} \sum_{j \in S_k}|\langle a_i, a_j \rangle|.
\end{equation}
Then we have:
\begin{equation}
    L_b = \frac{K}{n} \max_{k \in [K]}\|S^kA (S^k A)^T\| \leq \frac{K}{n} \max_{q \in [K]}\max_{i \in S_q} \sum_{j \in S_q}|\langle a_i, a_j \rangle| = \frac{K}{n}\mu_\ell(A, \Bar{S}, K).
\end{equation}
By definition of the SA factor, we can write:
\begin{equation}
    \Upsilon(A, \Bar{S}, K) = \frac{KL_f}{L_b} \geq \frac{\|A\|^2}{\mu_\ell(A, \Bar{S}, K)}.
\end{equation}
On the other hand, note that by the definition of the local accumulated coherence, we can have an upper bound for $\mu_\ell(A, \Bar{S}, K)$:
\begin{equation}
    \mu_\ell(A, \Bar{S}, K) := \max_{q \in [K]}\max_{i \in S_q} \sum_{j \in S_q}|\langle a_i, a_j \rangle| \leq \frac{n}{K} \max_{i \in [n]} \|a_i\|_2^2 = \frac{n}{K} \|A^T\|_{1 \rightarrow 2}^2,
\end{equation}
and hence we can have a relaxed lower bound for $\Upsilon(A, \Bar{S}, K)$:
\begin{equation}
    \Upsilon(A, \Bar{S}, K) \geq \frac{\|A\|^2}{\mu_\ell(A, \Bar{S}, K)} \geq \frac{K\|A\|^2}{n \|A^T\|_{1 \rightarrow 2}^2}.
\end{equation}
 Suppose that for some positive constant $\rho$ we have:
\begin{equation}
    \frac{\max_{i \in [n]} \|a_i\|_2^2}{\frac{1}{n}\sum_{j = 1}^n\|a_j\|_2^2} \leq \rho,
\end{equation}
then we can write:
\begin{equation}
   \max_{i \in [n]} \|a_i\|_2^2 \leq \rho \cdot \frac{1}{n}\sum_{j = 1}^n\|a_j\|_2^2 = \rho \cdot \frac{\|A\|_F^2}{n} = \rho \cdot \frac{\sum_{i = 1}^d \sigma(A^TA, i)}{n},
\end{equation}
and hence we can further lower bound $\Upsilon(A, \Bar{S}, K)$ by the cumulative eigenspectrum of the Hessian:
\begin{equation}
    \Upsilon(A, \Bar{S}, K) \geq \frac{\|A\|^2}{\mu_\ell(A, \Bar{S}, K)} \geq \frac{K\|A\|^2}{n \|A^T\|_{1 \rightarrow 2}^2} \geq \frac{K\cdot\sigma(A^TA, 1)}{\rho \cdot \sum_{i = 1}^d \sigma(A^TA, i)}.
\end{equation}
thus finishes the proof.
\end{proof}

\section{The proof of Theorem \ref{Thm_upper_bound}}\label{Ad6}

We present the proof of Theorem \ref{Thm_upper_bound} here.

\begin{proof}
\cite[Theorem 4.3.15]{horn2012matrix} indicates that, for a given Hermitian matrix $H \in \mb{R}^{n \times n}$, and any of its $m$-by-$m$ principal submatrices $H_m$, obtained by deleting $n - m$ rows and columns from $H$, we can have:
\begin{equation}
    \sigma(H_m, 1) \geq \sigma(H, n - m + 1).
\end{equation}
If we set $H_m = S^k A(S^k A)^T$, then we have:
\begin{equation}
    \|S^k A(S^k A)^T\| = \|S^k (A A^T) {S^k}^T\| \geq \sigma(AA^T, n - m + 1).
\end{equation}
Now we use the fact that $S^k A(S^k A)^T$ and $(S^k A)^TS^k A$ share the same non-zero eigenvalues, and meanwhile $AA^T$ and $A^TA$ also shares the same non-zero eigenvalues, we can have the following bound:
\begin{equation}
    \|(S^k A)^TS^k A\| = \|S^k A(S^k A)^T\| \geq \sigma(AA^T, n - m + 1) = \sigma(A^TA, n - m + 1).
\end{equation}
Then by the definition of $\Upsilon(A, \Bar{S}, K) $ we can obtain the upper bound.
\end{proof}

\section{The proof of Theorem \ref{Thm_rand par}}\label{Ad7}

We now present the proof of Theorem \ref{Thm_rand par}:

\begin{proof}
Suppose we randomly permute the index $[n]$ and generate the partition index $[S_1, S_2, ..., S_K]$. If we pick arbitrarily a number $k \in [K]$ where $S^k$ is the subsampling matrix, by the Matrix Chernoff inequality \cite{tropp2012user} we have the following relationship:
\begin{equation}
    \|A^T{S^k}^T S^k A\| \leq (1 + \delta_0) \cdot \frac{\|A\|^2}{K},
\end{equation}
with probability at least:
\begin{equation}
   P := 1 - d \cdot \left[ \frac{e^{\delta_0}}{ (\delta_0 + 1)^{\delta_0 + 1}} \right]^{\frac{\|A\|^2}{K \|A^T\|_{1 \rightarrow 2}^2}},
\end{equation}
 for any $\delta_0 > 0$. Now by choosing $\delta_0 = \delta \cdot \frac{K\|A^T\|_{1 \rightarrow 2}}{\|A\|^2}$ we can have following:
\begin{equation}
    \|A^T{S^k}^T S^k A\| \leq (1 + \delta \cdot \frac{K\|A^T\|_{1 \rightarrow 2}}{\|A\|^2}) \cdot \frac{\|A\|^2}{K} = \frac{\|A\|^2}{K} + \delta \|A^T\|_{1 \rightarrow 2}^2,
\end{equation}
with probability at least $P'$ where:
\begin{equation}
  1 - d \cdot \left[ \frac{e^\delta}{ (\delta\cdot \frac{K\|A^T\|_{1 \rightarrow 2}}{\|A\|^2})^{\delta}} \right] \geq P' := 1 - d \cdot \left[ \frac{e^\delta}{ \delta^{\delta}} \right].
\end{equation}
(This is because we restrict here $K \geq \frac{\|A\|^2}{\|A^T\|_{1 \rightarrow 2}^2}$.) Now by applying the union bound over all possible choices of $k$ and since we assume here $K \leq \min(n, d)$, we have, with probability at least $ 1 - d^2 \cdot \left[ \frac{e^\delta}{ \delta^{\delta}} \right]$:
\begin{equation}
    \max_{k \in [K]}\|A^T{S^k}^T S^k A\| \leq \frac{\|A\|^2}{K} + \delta \|A^T\|_{1 \rightarrow 2}^2.
\end{equation}
Then by definition:
\begin{equation}
    \Upsilon(A, \Bar{S}, K) = \frac{KL_f}{L_b} \geq \frac{\frac{K}{n}\|A\|^2}{\frac{K}{n}\left(\frac{\|A\|^2}{K} + \delta \|A^T\|_{1 \rightarrow 2}^2 \right)} = \frac{1}{\frac{1}{K} + \delta \frac{\|A^T\|_{1 \rightarrow 2}^2}{\|A\|^2}}
\end{equation}
Now since
\begin{equation}
   \|A^T\|_{1 \rightarrow 2}^2 = \max_{i \in [n]} \|a_i\|_2^2\leq \rho \cdot \frac{1}{n}\sum_{j = 1}^n\|a_j\|_2^2 = \rho \cdot \frac{\sum_{i = 1}^d \sigma(A^TA, i)}{n},
\end{equation}
we have:
\begin{equation}
    \Upsilon(A, \Bar{S}, K) \geq \frac{1}{\frac{1}{K} + \delta \frac{\|A^T\|_{1 \rightarrow 2}^2}{\|A\|^2}} \geq \frac{1}{\frac{1}{K} + \delta \rho \cdot\frac{\frac{\sum_{i = 1}^d \sigma(A^TA, i)}{n}}{\sigma(A^TA, 1)}}.
\end{equation}
Thus finishes the proof. 
\end{proof}

\section*{Acknowledgment}

We acknowledge the support from H2020-MSCA-ITN 642685 (MacSeNet) , ERC Advanced grant 694888, C-SENSE and a Royal Society Wolfson Research Merit Award. We thank Alessandro Foi, Vladimir Katkovnik, Cristovao Cruz, Enrique Sanchez-Monge, Zhongwei Xu, Jingwei Liang, Derek Driggs, Alessandro Perelli, Mikey Sheehan and Jonathan Mason for helpful discussions.

\bibliographystyle{siamplain}
\bibliography{main.bib}

\end{document}